\documentclass[10pt]{amsart}

\usepackage[active]{srcltx}

\usepackage{amscd,amsmath}
\usepackage{mathrsfs}
\usepackage{amsfonts}
\usepackage{amsaddr}
\usepackage{setspace}
\usepackage{cite}

\usepackage[margin=1.2in]{geometry}

\newtheorem{theorem}{Theorem}[section]
\newtheorem{lemma}[theorem]{Lemma}

\theoremstyle{definition}

\newtheorem{cor}[theorem]{Corollary}
\theoremstyle{remark}
\newtheorem{remark}[theorem]{Remark}

\newcommand{\comments}[1]{}

\renewcommand{\phi}{\varphi}
\newcommand{\R}{\mathbb R}

\newcommand{\D}{\displaystyle }
\newcommand{\Z}{\mathbb Z}
\newcommand{\N}{\mathbb N}

\newcommand{\cal}{\mathcal }

\newcommand{\bu}{\mathbf u}
\newcommand{\bv}{\mathbf v}

\newcommand{\fd}{\mathfrak D}
\newcommand{\bbv}{\mathbb V}
\newcommand{\ra}{\rightarrow}

\newcommand{\h}{\mathbb H}
\newcommand{\hh}{\dot{\h}}
\newcommand{\dt}{\frac{d}{dt}}

\newcommand{\cf}{\cal F}
\newcommand{\cs}{\cal S}
\newcommand{\cm}{\cal M}
\newcommand{\cl}{\cal L}
\newcommand{\hha}{\hh^{\alpha}_p}
\newcommand{\ati}{\alpha_{T_i}}

\newcommand{\eqn}{\begin{eqnarray}}
\newcommand{\een}{\end{eqnarray}}

\numberwithin{equation}{section}

\begin{document}

\title[Gevrey regularity for dissipative equations]{Gevrey regularity for a class of dissipative equations with analytic nonlinearity}

\author{\small Hantaek Bae$^\dagger$}
\address{\small Department of Mathematics, University of California, Davis, U.S.A}
\email{hantaek@math.ucdavis.edu}
\thanks{$^\dagger$\ {\em Corresponding author (Email: hantaek@math.ucdavis.edu)}}

\author{\small Animikh Biswas}
\address{\small Department of Mathematics and Statistics, University of Maryland, Baltimore County,
U.S.A.}
\email{abiswas@umbc.edu}
\thanks{The research of A. Biswas was supported in part by NSF grant DMS11-09532}




\begin{abstract}
In this paper, we establish Gevrey class regularity of solutions to a class of dissipative equations with an analytic nonlinearity in the whole space. This generalizes the results of Ferrari and Titi in the periodic space case  with initial data in $L^2-$based Sobolev spaces to the $L^p$ setting and in the whole space. Our generalization also includes considering rougher initial data, in negative Sobolev spaces in some cases including the Navier-Stokes and the subcritical quasi-geostrophic equations, and allowing the dissipation operator to be a fractional Laplacian.
Moreover, we derive global (in time) estimates in Gevrey norms which yields decay of higher order derivatives
which are optimal.
 Applications include (temporal) decay of solutions in  higher Sobolev norms for a large class of equations including the Navier-Stokes equations, the subcritical quasi-geostrophic equations, nonlinear heat equations with fractional dissipation, a variant of the Burgers' equation with a cubic or higher order nonlinearity, and the generalized Cahn-Hilliard equation. The decay results for the last three cases seem to be new while our approach provides an alternate proof for the recently obtained $L^p\, (1<p<2)$ decay result for the Navier-Stokes equations by Bae, Biswas and Tadmor. These applications follow from our global Gevrey regularity result for initial data in \emph{critical spaces}  with low regularity.
\end{abstract}

\maketitle


\section{INTRODUCTION}
It is known that regular solutions of many dissipative equations, such as the Navier-Stokes quations (NSE), the Kuramoto-Sivashinsky equation, the surface quasi-geostrophic equation,   the Smoluchowski equation,
the (periodic)  Cahn-Hilliard equation and Voight regularizations of Navier-Stokes and MHD equations are in fact  analytic \cite{masuda, ft1, biswas2007gevrey, dong, vuka2005, swanson11, larios10}.  In fluid-dynamics, the space analyticity radius has an important physical interpretation:  below this length scale, the viscous effects dominate the nonlinear effects and the Fourier spectrum decays exponentially  \cite{foias, hkr, hkr1, dt}. In other words,  the space analyticity radius yields a Kolmogorov type length scale  encountered in turbulence theory.  This fact concerning exponential decay can be used to show that the finite dimensional Galerkin approximations  converge exponentially fast; for instance, in the case of the complex Ginzburg-Landau equation,  analyticity estimates are used in \cite{doel-t} to rigorously explain  numerical observation  that the solutions to this equation can be accurately represented by a very low-dimensional  Galerkin approximation and the ``linear" Galerkin performs as well as the nonlinear one.  Other applications of analyticity radius  occur in establishing sharp temporal decay rates of solutions in higher Sobolev norms  \cite{oliver2000remark}, establishing geometric  regularity criteria for the Navier-Stokes equations, and in measuring the  spatial complexity of fluid flow \cite{kuk, kuk1, g}.

In this paper, we consider a nonlinear evolution equation of the form
\eqn  \label{eq:1.1}
u_t + \Lambda^{\kappa} u = G(u), \quad (\kappa > 1, t\ge 0), \quad  u_0 \in \cal{L},
\een
where the Banach space ${\cal L}$ considered here is a  potential space of adequate (positive or negative) regularity. The operator $\Lambda = (-\Delta)^{1/2}$ is assumed to be defined on dense subspace ${\cal D}(\Lambda) \subset {\cal L}$. The term $G(u)$ is a possibly nonlocal, analytic  function of $u$; see equation (\ref{eq:1.3}) for a precise description. The equation (\ref{eq:1.1}) will model dissipative evolutionary partial differential equations on the whole space $\R^d$ or on the periodic space $\mathbb{T}^{d}$. In this setting, we establish spatial analyticity  and  provide  decay  rates of  higher order norms of solutions to (\ref{eq:1.1}).

\subsection*{Analyticity}
It should be noted that the class of equations we consider encompasses many nonlocal equations (for instance, the Navier-Stokes and the quasi-geostrophic equations). Therefore, even the fact that the solutions are analytic for positive times, do not follow from Cauchy-Kovaleskaya type theorems. To overcome this difficulty, we take the mild solution approach initiated by Fujita and Kato \cite{fujita1964navier}, Giga and Miyakawa \cite{gm1985} and Weissler \cite{weis1980} for the Navier-Stokes equations, and study the evolution of Gevrey norms. The use of Gevrey norms was pioneered by Foias and Temam \cite{ft} for estimating space analyticity radius for the Navier-Stokes equations and was subsequently used by many authors (see \cite{cao1999navier, ferrari1998gevrey, biswas2007existence, biswas2010navier, biswas2012}, and the references there in); a closely related approach can be found in \cite{gk}. The Gevrey class approach taken here enables  one to avoid cumbersome recursive estimation  of higher order derivatives and is known to yield optimal estimates of the analyticity radius  \cite{oliver2000remark, optimaltiti}. (Other approaches to analyticity can be found in \cite{pav, miura2006, guber} for the 3D NSE, \cite{dong, dong1}  for the surface quasi-geostrophic equation, and \cite{ang} for certain nonlinear analytic semiflows.) We also note that most of the work has been considered in $L^{2}$-based Sobolev spaces. For example, in case the dissipation operator is $\Lambda^{2}=-\Delta$, equation (\ref{eq:1.1}) with an analytic nonlinearity  was  studied in \cite{ferrari1998gevrey}  and \cite{crt} for the periodic space case and the sphere case respectively. Their approach was based on energy technique and their space of initial data comprised of sufficiently smooth functions belonging to adequate ($L^2$-based) Sobolev spaces. By contrast, we work on $L^p$-based Banach spaces of initial data which allow us to consider much rougher (even distributional) initial data. In fact, the spaces we choose here, which are called \emph{critical spaces},  precisely correspond to a scale invariance property of the equation. In this paper, we show that solutions to (\ref{eq:1.1}) are {\em Gevrey regular}, i.e., they satisfy the estimate 
\[
{\D \sup_{0<t<T} \left\|e^{t^{1/\kappa} \Lambda}u \right\|_{\cal L}< \infty},
\]
locally in time for initial data of arbitrary size, and globally in time if the initial data is small  in critical spaces. Our main tool to show such estimates is a generalization of the Kato-Ponce inequality \cite{Kato2} to Gevrey spaces.

\subsection*{Decay Estimates}
The main emphasis in this paper is to obtain global (in time) Gevrey regular solutions to (\ref{eq:1.1}) for small initial data in critical regularity spaces. This has several applications in the study of long term dynamics. It turns out (as we show here) that many of the equations encountered in fluid dynamics has the property that for large times, the solutions have small norm in these critical regularity spaces. Thus,  as a consequence of our result, we obtain exponential decay of Fourier coefficients in the periodic case and algebraic decay of higher order $L^p$ based Sobolev norms in the whole space for a wide class  of equations including the Navier-Stokes equations, sub-critical quasi-geostrophic equation,  a variant of Burgers' equation with a higher order polynomial nonlinearity, the generalized Cahn-Hilliard equation  and a nonlinear heat equation.  In some cases (for example, the Navier-Stokes, the subcritical quasi-geostrophic equations and nonlinear heat equations with fractional dissipation) this generalizes known results  while in some others (for example, Burgers' equation with cubic or higher order nonlinearity and the Cahn-Hilliard equation),  the results are new to the best of our knowledge.  For instance, for the Navier-Stokes equations,  in \cite{oliver2000remark}, sharp upper and lower bounds for the (time) decay of $L^2$-based higher order Sobolev norms in  the whole space $\R^d$ was established for a certain class of initial data. Our result yields  decay for $L^p$-based ($p>1$) higher order homogeneous Sobolev norms for the Navier-Stokes equations, and for a larger class of initial data. Although the $L^p$ result for $p>2$ can be deduced from the $L^2$ case by Sobolev inequality,  this is not the case for $L^p$ with $1<p<2$. The $L^p$ decay results for $1<p <2$ in  the whole space $\R^d$   were obtained only recently in \cite{biswas2012}. Our approach here provides an alternate proof of this result while avoiding the sophisticated  machinery used in \cite{biswas2012}.  As another corollary of our general result, we also recover a similar decay result in \cite{dong} for the quasi-geostrophic equations with a different proof that avoids the iterative estimation of higher order derivatives.  In summary, our main results provide an unified approach to a variety of decay results, some existing and some new \cite{oliver2000remark, schon1995, schonwieg, dong},  thus generalizing them to: a wider class of equations; to $L^p$ decay; and allowing for a larger class of initial data. A more detailed comparison of our results {\em vis a vis} some known results is made in the remarks subsequent to the relevant theorems.

The organization of this paper is as follows. In Section 2, we state our main results while in Section 3, we state our main applications concerning decay. Sections 4 and 5 are devoted to the proofs of these results while in the Appendix, we have included some requisite background on Littlewood-Paley decomposition of functions.

\section{MAIN RESULTS} \label{sec:2}

Before describing our main results, we start by  establishing some notation and concepts.

A function $u \in C([0,T];{\cal L})$ is said to be a \emph{strong} solution of (\ref{eq:1.1}) if
$\partial_t u $ exists and $u \in {\cal D}(\Lambda^\kappa)\, a.e. $  and the equation (\ref{eq:1.1}) is satisfied $a.e.$ A \emph{mild} solution of (\ref{eq:1.1}) is a solution  of the corresponding integral equation
\eqn  \label{eq:1.2} 
u(t)=(Su)(t), \quad (Su)(t):= e^{-t\Lambda^{\kappa}}u_0 + \int_0^{t} e^{-(t-s)\Lambda^{\kappa}}G(u(s)) ds,
\een
where $u$ is assumed to belong to $C([0,T];{\cal L})$.  The integral on the right hand side of (\ref{eq:1.2}) is interpreted as a Bochner integral. Henceforth, by a solution to (\ref{eq:1.1}) we will mean a mild solution, which is a fixed point of the map $S$. For a discussion on the connection between weak, strong, mild and classical solutions see \cite{sell2002}.

For each $i=0,\cdots ,n$, let  $T_i$ be a bounded operator mapping $\hh^{\alpha + \ati}_p, \ati \ge 0$ to $\hha $ such that $T_i$ commutes with $\Lambda $ and its operator  norm is bounded uniformly with respect to $\alpha $ and $p$. More precisely, we assume that there exists  constants $C>0$ and $\ati \ge 0$ such that for all $\alpha \in \R, 1 < p < \infty $ and $ v \in \hha$, we have
\eqn \label{eq:2.1}
\left\|T_i v\right\|_{\hha} \le C \|v\|_{\hh^{\alpha + \ati}_p}, \quad T_i \Lambda v = \Lambda T_i v.
\een
Examples of such operators are Fourier multipliers with symbols
$$
m_i (\xi) = \Omega_i(\xi)\xi^{\vec{\alpha}_{T_i}} \quad \text{or} \quad m_i (\xi) = \Omega_i(\xi)|\xi|^{\ati},
$$
where $\Omega_i(\cdot)$ are bounded homogeneous functions of degree zero (i.e., $\Omega_i(c\xi) = \Omega(\xi), c \in \R$) and $\vec{\alpha}_{T_i} = \left(\ati^{(1)}, \cdots , \ati^{(d_1)}\right), \ \ati^{(j)} \ge 0$. Other examples also include shift operators on
the space of distributions (in which case $\ati =0$).
\comments{for the first case in (\ref{eq:2.1}) while  $\ati \ge 0$ for the second case in (\ref{eq:2.1}). We will also henceforth denote  $\ati =\sum_{j=1}^{d_{1}}\alpha^{(j)}_{T_i}$  for the first case in (\ref{eq:2.1}), and .}

The nonlinearity $G$ that we consider has the form
\eqn \label{eq:1.3}
G(u) = T_{0} F\left(T_1 u,T_2 u,\cdots,T_n u\right), \quad F(z_1, \cdots, z_n)= \sum_{\alpha \in \Z_+^n}a_{\alpha}z^{\alpha},
\een
where $F$ is an analytic function defined on a neighborhood of the origin in $\R^n$. More specific
assumptions on $F$ will be made in subsequent sections.

\subsection{A homogeneous degree $n$ nonlinearity}
In this section, we assume that $F$ is a monomial of degree $n$, i.e.,
\eqn \label{eq:2.2}
G(u) = T_{0} F \left(T_1 u,T_2 u,\cdots,T_n u \right), \quad F(z_1, \cdots, z_n)=  z_1 \cdots z_n,
\een
for some $n \ge 2, n \in \N$. However, see Remark \ref{rem:extension} for an extension to a general homogeneous polynomial of degree $n$. The homogeneous potential spaces are defined by
\[
\hh^{\alpha}_p = \left\{f \in {\cal S}'(\R^d): \Lambda^{\alpha}f \in L^p(\R^d), \  \|f\|_{\hh^{\alpha}_p}:=\|\Lambda^{\alpha}f\|_{L^p} < \infty \right\}, \quad \alpha \in \R.
\]
For $\xi \in \R^d$, denote $\|\xi\|_1 = \sum_{i=1}^d |\xi_i|$ while $\|\xi\|=\big(\sum^{d}_{i=1} \xi_i^2\big)^{1/2}$ denotes the usual Euclidean norm on $\R^d$. Recall that the norms $\|\cdot \|_1$ and $\|\cdot\|$ on $\R^d$ are equivalent. Let $\Lambda_1$ be a Fourier multiplier whose symbol is  given by $m_{\Lambda_1}(\xi)=|\xi|_1$. Choose (and fix) a constant $c>0$ such that  $c|\xi|_1 < \frac{1}{4}|\xi|$ for all $\xi \in \R^d$. This is possible since all norms on $\R^d$ are equivalent.  With this notation, we define the  Gevrey norm in $\hh^{\beta}_p$ to be
\begin{gather}  \label{eq:2.3}
\begin{split}
\left\|v\right\|_{Gv(s,\beta,p)} =  \left\|e^{cs^{1/\kappa}\Lambda_1}\Lambda^{\beta}v \right\|_{L^p}= \left\|e^{cs^{1/\kappa}\Lambda_1}v \right\|_{\hh^{\beta}_p}, \quad s \ge 0, \ \ \beta \in \R,  \ \ 1\le p \le \infty. 
\end{split}
\end{gather}
Before stating our result here, we need to determine an appropriate function space for the initial data.
The idea is to choose a function space where the linear term $u_t + \Lambda^{\kappa} u$ and the nonlinear term $G(u)$ have the same regularity. This can be interpreted in terms of the scaling invariance. The equation (\ref{eq:1.1}) with the nonlinearity (\ref{eq:2.2}) satisfies the following scaling. Assume that $u$ is a solution of (\ref{eq:1.1}). Then, the same is true for the rescaled functions
\[
u_{\lambda}(t,x)=\lambda^{s}u\left(\lambda^{\kappa}t, \lambda x\right), \quad s=\frac{1}{n-1} \left(\kappa-\sum^{n}_{i=0}\alpha_{T_{i}}\right).
\]
Therefore, $\hh^{\beta_{c}}_p$ is the scaling invariant space for initial data, with
\eqn \label{eq:2.12}
\beta_{c}:=\frac{d}{p} -\frac{1}{n-1} \left(\kappa-\sum^{n}_{i=0}\alpha_{T_{i}}\right)
\een
satisfying that
\eqn
u_{\lambda0}=\lambda^{s}u_{0}(\lambda x), \quad \left\| u_{0}\right\|_{\hh^{\beta_{c}}_p}= \left\| u_{\lambda0}\right\|_{\hh^{\beta_{c}}_p}.
\een
We can expect the local existence for $\beta>\beta_{c}$ for large data and the global existence for $\beta=\beta_{c}$ for small data. Recall also that for $\beta < \frac{d}{p}$, $\hh^{\beta}_p(\R^d)$ is a Banach space while for $\beta \ge \frac{d}{p}$, it is a normed space which is not complete  \cite{Danchin, bahourifourier}.

\begin{theorem} \label{thm:2.8}
Let $G$ be a nonlinearity as in (\ref{eq:2.2}). Assume that the following condition holds:
\eqn \label{eq:2.13}
\sum_{i=0}^n\ati \le \kappa, \quad  \min_{1\le i \le n}\ati > \max\left\{ \frac{\sum_{i=1}^n\ati}{n} - \frac{d}{np}, \ \   \frac{\sum_{i=0}^n\ati -\kappa}{n-1} \right\}.
\een
Let  $u_0 \in \hh^{\beta_0}_p$, with
\[
\frac{d}{p} - \frac{\kappa-\sum_{i=0}^n\ati }{n-1} =\beta_c \le \beta_0 < \min\left\{\frac{d}{p}, \frac{d}{p} + \min_{1\le i \le n}\ati\right\}.
\]
Then there exists $T=T(u_0)>0$ and $\beta >0$ with $\beta_0+\beta >0$, and a solution of (\ref{eq:1.2}) belonging to the space $C([0,T]; \hh^{\beta_0}_p)$ which additionally satisfies
\eqn  \label{eq:2.14}
\max \left\{\sup_{o<t <T} \|u(t)\|_{Gv(t,\beta_0,p)}, \hspace{0.2cm} \sup_{0<t<T}t^{\beta/\kappa}\|u(t)\|_{Gv(t,\beta_0+\beta,p)} \right\} \le  2\|u_0\|_{\hh^{\beta_0}_p}.
\een
Moreover, if $\beta_0 > \beta_c$,  the time of existence $T$ is given by
\[
T \ge C/ \|u_0\|_{\hh^{\beta_0}_p}^{(\beta_0-\beta_c)/\kappa},
\]
for some constant $C$  independent of $u_0$. On the other hand, in case $u_0 \in \hh^{\beta_c}_p$, there exists $\epsilon >0$ such that whenever  $\|u_0\|_{\hh^{\beta_c}_p}<\epsilon$, we can take $T= \infty$.
\end{theorem}

\begin{remark}  \label{rem:extension}
The above theorem can be readily generalized to the case of a nonlinearity where $F$ in  (\ref{eq:2.2}) is a homogeneous polynomial of degree $n$ with the following property:
\begin{equation*}  
\begin{split}
& F(z_1, \cdots, z_n) = \sum_{\alpha \in {\cal S}}a_{\alpha}z^{\alpha}, \quad \text{where}\\
& {\cal S} =\left\{\alpha=(\alpha_1,\cdots,\alpha_n) \in \Z_+^N: \sum_{i=1}^n \ati \chi_{\{\alpha_i \neq 0\}} = \alpha\ \mbox{and}\ \sum_{i=1}^n \alpha_i =n\right\}.
\end{split}
\end{equation*}
This for instance is satisfied if $F$ is a homogeneous polynomial of degree $n$ and $\ati = \alpha_{T_j}$ for all $i,j \neq 0$.
\end{remark}

\begin{remark}
A version of Theorem \ref{thm:2.8}, for the special case of a quadratic nonlinearity, was established in  \cite{biswas}. In contrast to the set up of the real space here, the norms on the initial data space there were defined in the Fourier space. This enables one to completely avoid the detailed harmonic analysis machinery used in the proof here. However, due to the Hausdorff-Young inequality, even restricted to the quadratic nonlinearity case, our consideration here yields a larger space of initial data in several applications. As we will see later, due to this, we can obtain  decay of   $L^p$ ($p>1$) based higher Sobolev norms (for instance for the Navier-Stokes equations) not available in \cite{biswas}.
\end{remark}

\subsection{Analytic Nonlinearity}
In this section, we consider the more general case of an analytic nonlinearity. Let 
\[
F(z)= \displaystyle\sum_{k \in \Z_+^n}a_k z^k
\]
be a real analytic function in a neighborhood of the origin. Here $z=(z_1,\cdots,z_n) \in \R^n$ and we employ the multi-index convention $z^k = z_1^{k_1}\cdots z_n^{k_n}$ for $k=(k_1,\cdots,k_n)$.  The ``majorizing function" for $F$ is defined to be
\eqn \label{eq:4.1}
F_M(r) = \sum_{k\in \Z_+^d}|a_k|r^{|k|}, \quad |k|=k_1 + \cdots +k_d, \ \ r<\infty.
\een
The functions $F$ and $F_M$ are clearly analytic in the open balls (in $\R^d$ and $\R$ respectively) with  center zero and radius
\eqn \label{eq:4.2}
R_M= \sup \left\{r:F_M(r) < \infty\right\}.
\een
We assume that the set in the right hand side of (\ref{eq:4.2}) is nonempty. The derivative of the function $F_M$, denoted by $F_M'$, is also analytic in the ball of radius $R_M$. We now consider the nonlinearity $G$  of the type
\eqn \label{eq:4.3}
G(u) = T_0F \left(T_1 u, \cdots,T_n u \right),
\een
where $T_i$ are as defined in (\ref{eq:2.1}), and the inhomogeneous Sobolev space
\[
\h^{\alpha}_p = \left\{f:\R^{d} \ra \R^{d_1}: \|f\|_{\h^{\alpha}_p}:=\|(I+\Lambda)^{\alpha}f\|_{L^p} < \infty \right\}.
\]
We recall that from the standard Sobolev inequalities and (\ref{eq:2.7}),  we have
\eqn \label{eq:4.4}
\quad \quad \|f\|_{L^{\infty}} \le \|f\|_{\h^{\beta}_p}, \quad \|fg\|_{\h^{\beta}_p} \le
C\|f\|_{\h^{\beta}_p}\|g\|_{\h^{\beta}_p} \quad \text{for} \hspace{0.2cm}\beta > \frac{d}{p}, \hspace{0.2cm}1<p < \infty.
\een
The corresponding Gevrey norm  is defined as
\eqn \label{eq:4.5}
\|v\|_{Gv(s,\beta,p)} =  \left\|e^{\frac{1}{2}s^{1/\kappa}\Lambda_1}(1+\Lambda)^{\beta} v \right\|_{L^p}.
\een
We will show that the Gevrey space $Gv(s,\beta,p)$ with $\beta>\frac{d}{p}$ is a Banach algebra in Lemma \ref{lem:4.1}. We are now ready to state our main result concerning analytic nonlinearity.

\begin{theorem}   \label{thm:4.3}
Let $\beta > \frac{d}{p}$ and assume that $\frac{\alpha_{T_0}+ \alpha }{\kappa} < 1$, where $\alpha := \displaystyle \max_{1 \leq i \leq n}\{\ati\} $. Let $ \|u_0\|_{\h^{\beta+\alpha}_p} < R$. Assume that $2RC < R_M$ where $C$ as in Lemma \ref{lem:4.2} and $R_M$ as in (\ref{eq:4.2}).  There exists a time $T>0$ and a solution $u$ of (\ref{eq:1.2}) such that
\[
\sup_{0<t<T} \left\|u(t)\right\|_{Gv(t,\beta,p)} < \infty.
\]
\end{theorem}

Note that unlike Theorem \ref{thm:2.8}, we do not get a global existence result here even in case of small initial data. However, in the periodic space case, we do obtain a global existence result for small data. We need to assume however that the nonlinearity $G$ has the property that it leaves the space of mean zero periodic functions invariant (this happens for instance if $T_0 = \nabla $). This is due to the fact that in the periodic space, $\Lambda$ has a minimum eigenvalue, denoted by $\lambda_0 >0$, and the Fourier spectrum of all periodic functions with space average zero is contained in the complement of a ball with radius $\lambda_0$. More precisely, we have the following result.

\begin{theorem}  \label{thm:4.5}
Consider the equation (\ref{eq:1.1}) in the periodic space. Assume that the space of mean zero functions are invariant under  $G$ and the Fourier coefficients  of $G$ satisfy 
\[
a_0 =0, \quad \sum_{k \in \Z_+^d, |k|\ge 1}|a_k| < \delta ,
\]
where $\delta \ge 0$ is suitably small. Then, there exists $\epsilon >0$ such that if
$\|u_0\|_{\h^{\beta}_p} < \epsilon$, with $ \beta > \frac{d}{p}$, and $1<p < \infty$, we can obtain an unique solution to (\ref{eq:1.2}) satisfying 
\[
\sup_{0<t< \infty}\|u(t)\|_{Gv(t,\beta,p)} < \infty.
\]
\end{theorem}

\subsection{Equation in Fourier space}
In Section 2.2, we worked under the restriction $1<p<\infty$. In this section, we consider the case $p= \infty $. The harmonic analysis tools used in the previous sections do not work here  since the singular integrals are not bounded in $L^\infty $. We therefore resort to working in the frequency space using the Fourier transform. Recall that if the Fourier transform of a distribution  is in $L^{1}$ in the Fourier spaces, then it is an $L^{\infty}$ function. The development here is in the spirit of \cite{biswas2007gevrey} and  \cite{biswas2010navier}. The other borderline case of $p=1$ (in space variables) is similar and discussed in Remark \ref{rem:otherborder}.

We denote by $\cf $ the Fourier transform (in the space variables) and by $\cf^{-1}$ its inverse. By a notational abuse, letting $u = \cf(u)$, we can reformulate (\ref{eq:1.1}) as
\eqn \label{eq:4.10}
u_t(\xi,t) + |\xi|^{\kappa} u(\xi,t) = G\left(u(\cdot,t)\right)(\xi), \quad u(\xi,0)=u_0(\xi),
\een
where $G$ is as in (\ref{eq:4.3}). Recall that the Fourier transform converts products in real space to convolutions in the frequency space and  the analytic function $F$ in (\ref{eq:4.1}) takes the form
\[
F(v)= \displaystyle\sum_n a_n v_1^{\ast n_1}\cdots v_d^{\ast n_d}, \quad v=\left(v_1,\cdots,v_d\right)
\]
where $\ast$ denotes convolution. Denoting by $\fd$ the multiplication operator $(\fd v)(\xi) = |\xi|^{\kappa}v(\xi)$, we can write the mild formulation of (\ref{eq:1.1}) as
\eqn \label{eq:4.11}
u(t) = e^{-t\fd}u_0 + \int_{0}^{t} e^{-(t-s)\fd}G\left(u(s)\right) ds.
\een
In this section, we denote $\|\cdot\|$ the $L^1$ norm in the Fourier space, i.e.,
\[
\|v\|=\int_{\R^{d}} \left|v(\xi)\right|d\xi.
\]
We also recall that $L^1$ is a Banach algebra under convolution such as
\[
\|u \ast v\| \le \|u\|\|v\|.
\]
For $s \ge 0$ and $\beta \in \R$, we now introduce the Gevrey norms as
\eqn \label{eq:4.12}
\|v\|_{Gv(s,\beta)} = \int_{\R^{d}} e^{\frac{1}{2}s^{1/\kappa}|\xi|}|\xi|^{\beta}|v(\xi)|d\xi.
\een
For notational simplicity, we suppress the dependence of the Gevrey norm on $\kappa$ (since it is fixed), and  we denote $\|v\|_{Gv(0,\beta)} = \|v\|_{\beta}$ when $s=0$ and we write $\|v\|_{Gv(s,0)}=\|v\|_{Gv(s)}$ when $\beta =0$. Also, denote
\[
\bbv_\beta = \left\{v: \|v\|_{\beta} < \infty \right\}.
\]
When $\beta =0$, we simply write $\bbv=\bbv_0$. For $u=(u_1,\cdots,u_{d}), v=(v_1,\cdots,v_{d})$ vector valued functions, we denote
\[
u \ast v = \Big(\sum_{ij} b_{ijk}u_i \ast v_j \Big)_{k=1,2, \cdots, d}, \quad b_{ijk} \in \R.
\]
It is easy to see, applying the Cauchy-Schwartz inequality,  that 
\[
\left|(u \ast v)(\xi)\right| \le C \left(|u| \ast |v|\right)(\xi), \quad C= \max_{i,j,k}|b_{ijk}|.
\]
We now state the existence of a solution to (\ref{eq:4.10}) with respect in the Gevrey spaces (\ref{eq:4.12}).

\begin{theorem}  \label{thm:4.10}
Let $\alpha=\displaystyle\max_{1\leq i\leq n} \{\ati\}$ and assume that
\[
\frac{\alpha_{T_0}+\alpha }{\kappa} < 1, \quad \max\{ \|u_0\|, \|u_0\|_{\alpha} \}< R, \quad 2RC < R_M,
\]
where $C$ is as in Lemma \ref{lem:4.8} and $R_M$ as in (\ref{eq:4.2}).  There exists a time $T>0$ and a solution $u$ of (\ref{eq:4.10}) such that 
\eqn  \label{globalgevrey}
\sup_{0<t<T}\|u(t)\|_{Gv(t)} < \infty.
\een 
Additionally, suppose that we consider the problem in the periodic space and the nonlinearity $G$ leaves the space of functions with zero space average invariant. Then, if we consider (\ref{eq:4.10}) on that subspace, there exists $\epsilon >0$ such that we can take $T=\infty $ in (\ref{globalgevrey}) provided $ \max\{ \|u_0\|, \|u_0\|_{\alpha} \}< \epsilon$.
\end{theorem}

\begin{remark}  \label{rem:otherborder}
In this section, we worked with $L^1$ norm of the Fourier transform of the function. Due to the Hausdorff-Young inequality, this corresponds to the $L^\infty $ norm in the space variable. However, the other borderline case of Theorem \ref{thm:4.3} occurs when we have $L^1 $ norm in the space variable. In the frequency space, this corresponds to the $L^\infty$ norm (in the sense that $\|{\cal F}(u)\|_{L^{\infty}} \le \|u\|_{L^1}$). We can obtain an analogue of Theorem \ref{thm:4.10} if instead of the $L^1$ norm, we let
\[
|\|u\||= \sup_{\xi \in \R^d} \left(1+|\xi| \right)^{\alpha_0}|u(\xi)|, \ \ \alpha_0 > d.
\]
The conditions with $\max\{ \|u_0\|, \|u_0\|_{\alpha} \}$ in Theorem \ref{thm:4.5}  are replaced with analogous conditions with $|\|u_0\||_{\alpha + \alpha_0}$. The proof of this fact is very similar to the proof of Theorem \ref{thm:4.10} once one notes that like the $L^1$ space, the Gevrey space based  on this norm is also a Banach algebra under convolution (see \cite{biswas2010navier}).
\end{remark}

\begin{remark}[Exponential decay of Fourier coefficients]
In the periodic  space setting (with period $L$),  the operator $\Lambda$ restricted to the subspace of functions with space average zero, has a discrete spectrum with its lowest eigenvalue being $\frac{2\pi}{L}$. Then, provided $ \max\{ \|u_0\|, \|u_0\|_{\alpha} \}$ in small enough (or in view of Remark \ref{rem:otherborder}, if $|\|u_0\||_{\alpha_0+\alpha}$ is small enough), we have the exponential decay of the Fourier coefficients
\[
\left|{\cal F}(u)(k,t) \right| \le Ce^{-t^{1/\kappa}|k|}, \quad k\in \mathbb{Z}^{3}_{+}. 
\]
This is a generalization of the results in \cite{foias} and \cite{dt} for the special case of the
3d Navier-Stokes equations. In fact, for the 3d Navier-Stokes equations, following similar techniques as is presented here, one can sharpen this decay result (see \cite{biswas2007gevrey}) by demanding only that
\[
\sup_{k \in \Z_+^3} |k|^2|{\cal F}(u_0)(k)| < \epsilon .
\]
This fact was proven in \cite{biswas2007gevrey} using similar methods, and independently in \cite{Sinai}.
\end{remark}

\section{APPLICATIONS: DECAY OF SOBOLEV NORMS} \label{sec:3}

Theorem \ref{thm:2.8} tells us that if the initial data are small in  critical spaces  $\hh^{\beta_c}_p$, the solutions are globally in the Gevrey class. Due to Lemma \ref{lem:2.1} (or Lemma \ref{lem:2.2}), this allows us to obtain the following time decay of homogeneous Sobolev norms: for $\zeta >\beta_c$,
\begin{gather} \label{gevdecay}
\begin{split}
\left\|\Lambda^{\zeta} u(t)\right\|_{L^{p}} &=\left\|\Lambda^{\zeta -\beta_c}e^{-ct^{1/\kappa}\Lambda_1} e^{ct^{1/\kappa}\Lambda_1}\Lambda^{\beta_c} u(t) \right\|_{L^{p}} \\
&\le C_{\zeta}  t^{-\frac{1}{\kappa}(\zeta-\beta_{c})} \left\|u(t)\right\|_{Gv(t,\beta_c,p)}, \ \ C_\zeta \sim C^\zeta \zeta^\zeta.
\end{split}
\end{gather}
On the other hand, if we can show that a solution $u$ of (\ref{eq:1.1}) satisfies
\eqn   \label{eq:3.1}
\liminf_{t\ra \infty} \left\|u(t)\right\|_{\hh^{\beta_c}_p} =0,
\een
then due to Theorem \ref{thm:2.8},  after a certain transient time $t_{0}$, we have
\[
\sup_{t > t_0} \left\|u(t)\right\|_{Gv(t-t_0,\beta_c,p)} < \infty .
\]
Consequently, we obtain
\eqn  \label{eq:3.2}
\left\|\Lambda^{\zeta} u(t) \right\|_{L^{p}} \leq C_{\zeta} \left\|u(t_{0})\right\|_{\hh^{\beta_0}_p} (t-t_{0})^{-\frac{1}{\kappa}(\zeta-\beta_{c})},  \quad \zeta>\beta_{c}
\een
where $\|u(t_{0})\|_{\hh^{\beta_c}_p}$ is sufficiently small to apply Theorem \ref{thm:2.8}. In the next sections, we will provide several examples where this can be achieved.

\subsection{Navier-Stokes equations}
The first application of Theorem \ref{thm:2.8} is the three dimensional Navier-Stokes equations: 
\begin{subequations} \label{ns}
\begin{align}
& u_{t} +u\cdot \nabla u -\mu\Delta u+\nabla p=0, \\
& \nabla \cdot u=0.
\end{align}
\end{subequations}
The critical space for (\ref{ns}) is $\hh^{\beta_{c}}$ with $\beta_{c}=\frac{3}{p}-1$. Compared with previous works by \cite{schon1985, schon1991, schon1995, oliver2000remark, schonwieg, Miyakawa} and others, where decay of $L^{2}$-based Sobolev norms have been achieved, we provide decay of $L^{p}$-based Sobolev norms.

\begin{theorem} \label{thm:3.1}
Let $1 < p < \infty$ and  $u$ be a weak solution of the three dimensional Navier-Stokes equations with $u_0 \in L^2$. In case  $1 < p < 2$, we additionally assume that $\omega_0 = \nabla \times u_0 \in L^1$. Let $\beta_c = \frac{3}{p}-1$.  Then, we have the following decay estimate:
\eqn   \label{eq:3.3}
\left\|\Lambda^{\zeta} u(t) \right\|_{L^p} \le C_{\zeta} \left\|u(t_{0})\right\|_{\hh^{\beta_c}_p} \left(t-t_0\right)^{-\frac{1}{2}(\zeta -\beta_c)}, \quad \zeta> \max \left\{0, \beta_{c}\right\},
\een
where $C_\zeta \sim C^\zeta\zeta^\zeta$ for $p=2$ and $\left\|u(t_{0})\right\|_{\hh^{\beta_c}_p}$ is sufficiently small to apply Theorem \ref{thm:2.8}.
\end{theorem}

\begin{remark}
Existence of solutions to the (\ref{ns}) in Gevrey classes was first proven for the periodic space case  by Foias and Temam (\cite{ft}) (for initial data in $\h^1$) and subsequently by Oliver and Titi on the whole space with initial data in $\h^{s}, s > n/2, n=2,3$ (see also \cite{LR} for initial data in $\hh^{1/2}$ for (\ref{ns})). By following a slightly different approach, namely interpolating the $L^p$ norms of the solution and its analytic extension, Grujic and Kukavica \cite{gk} proved analyticity of solutions to the (\ref{ns}) for initial data in $L^p, p >3$ (see also \cite{LR1} for a different proof, and \cite{pav} and \cite{guber} for initial data in the larger critical space $BMO^{-1}$). Analyticity of solutions for initial data in homogeneous potential spaces $\hha , 1<p<\infty, \alpha \ge \frac{3}{p}-1$, which includes the above mentioned $L^p$ spaces,  follows from Theorem \ref{thm:2.8}. The decay in $L^2$-based (homogeneous) Sobolev norms for (\ref{ns})  were, to the best of our knowledge, first given in \cite{schon1995} and \cite{schonwieg}. However, the constants $C_{\zeta}$ were not explicitly identified there. The sharp (and optimal, in the sense of providing lower bounds as well) decay results were provided by Oliver and Titi (\cite{oliver2000remark}) following the Gevrey class approach. The constants $C_{\zeta}$ in (\ref{eq:3.3}) is of the same order as in  \cite{oliver2000remark}. Thus, (\ref{eq:3.3}) is a $L^p$ version of the sharp decay result in \cite{oliver2000remark}. In  \cite{schon1995} and \cite{oliver2000remark}, there is also an assumption of the decay of the $L^2$ norm of the solution. This is circumvented  here due to our working in the ``critical" space $\hh^{1/2}$.
\end{remark}

\subsection{Subcritical dissipative surface quasi-geostrophic equations}
These equations are given by
\begin{subequations} \label{qg}
\begin{align}
& \eta_t + u \cdot \nabla \eta + \Lambda^\kappa \eta =0, \quad (x,t) \in \R^2 \times (0,\infty ),\\
& u = \left(-{\cal R}_2\eta, {\cal R}_1\eta \right) := \left(-\partial_{x_2}\Lambda^{-1} \eta, \partial_{x_1}\Lambda^{-1} \eta\right), \\
& \eta(x,0) = \eta_0(x).
\end{align}
\end{subequations}
We consider the sub-critical case when the parameter $\kappa $ satisfies $\kappa \in (1,2]$. For initial data in $L^p, p \ge \frac{2}{\kappa -1}$, the global existence of a unique, regular solution to (\ref{qg}) in the sub-critical case  is known (see \cite{Carrillo} and the references there in). The following theorem concerning the long time behavior of higher Sobolev norms was first given in \cite{dong}. The proof in \cite{dong} involves iterative estimation of higher order derivatives involving elaborate combinatorial arguments. We obtain this result as an application of Theorem \ref{thm:2.8}. The use of Gevrey class technique presented here provides an alternate proof.

\begin{theorem}  \label{thm:qgdecay}
Let $\kappa \in (1,2]$ and denote $p_0=\frac{2}{\kappa -1}$. Let $\eta $ be the unique regular solution to (\ref{qg})  for initial data $\eta_0 \in L^{p_0}(\R^2)$. Then, the following estimate holds 
\begin{gather}  \label{qgdecay}
\|\eta\|_{\hh^{\alpha}_{p_0}} \le \frac{C^{\alpha} \alpha^{\alpha}}{t^{\frac{\alpha}{\kappa}}} \ \ \mbox{for all} \ t>0, \ \alpha >0,
\end{gather}
where the constant $C$ may depend on $\eta_0$, but is independent of $t, \alpha$.
\end{theorem}

\subsection{Nonlinear Heat Equation with Fractional Dissipation}
We consider the nonlinear heat equation of the form
\begin{gather}   \label{nonlinheat}
u_t + \Lambda^\kappa = \alpha |u|^{(n-1)}u, \quad x \in \R^d, \ t \ge 0, \ \alpha \in \R.
\end{gather}
Here we take $\kappa , n \in (1,\infty)$. For all initial data $u_0 \in \hh^{\beta}_2, \ \beta \ge \beta_c=\frac{d}{2} - \frac{\kappa}{n-1}$, local in time analyticity of the solution follows immediately from Theorem \ref{thm:2.8}. We give an application of Theorem \ref{thm:2.8} to the decay of solutions.

\begin{theorem}   \label{thm:nonlinheat}
Let $u$ be a solution of (\ref{nonlinheat}) with space periodic boundary condition. Let
$|\alpha | < \epsilon $ for a suitably small $\epsilon $ and assume that
either (i) or (ii) below holds:
\begin{gather}
(i)\  \kappa < d\ \mbox{and}\ n \le \frac{d+\kappa}{d-\kappa}\quad (ii)\ \kappa \ge d,\ n \in (1,\infty).
\end{gather}
In this case,
\[
\left\|\Lambda^{\zeta} u(t)\right\|_{L^{2}} \leq C_{\zeta} \left\|u(t_{0})\right\|_{\hh^{\beta_c}} (t-t_{0})^{-\frac{1}{2}(\zeta-\beta_{c})}, \quad \zeta>\beta_{c}:=\frac{d}{2}-\frac{\kappa}{n-1},
\]
where $\left\|u(t_{0})\right\|_{\hh^{\beta_c}}$ is sufficiently small to apply Theorem \ref{thm:2.8}. In the whole space case, the same conclusion remains valid if $\kappa < d\ \mbox{and}\ n =\frac{d+\kappa}{d-\kappa}$.
\end{theorem}

\begin{remark}
Local in time analyticity result for this equation can be found in \cite{gk2, Gramchev} while some asymptotic properties of solutions (particularly, involving the $L^\infty$ norm of the solution) can be found in \cite{gmira}, \cite{escobedo}, \cite{v2} and references there in. However, the higher order decay results for (\ref{nonlinheat}) appear to be new.
\end{remark}

\subsection{Burgers-type equation with higher order nonlinearity}
We now give an application  to decay for a viscous Burgers' equation on $\R$  of the form
\eqn \label{eq:3.17}
\partial_t u - \Delta u = \partial_x (u^n), \quad n \ge 3.
\een
Here we take $p=2$ and, as is customary, we denote $\hh^{\beta}_2 = \hh^\beta, \beta \in \R, \beta \neq 0$ and $\hh^0 = L^2$. For $p=2, \kappa=2, \alpha_{T_0}=1$ and $ \alpha_{T_i}=0, 1\le i \le n$, the critical space for (\ref{eq:3.17}) is $\hh^{\beta_{c}}$ with $\beta_c=\frac{1}{2}-\frac{1}{n-1}$.

\begin{theorem} \label{thm:3.2}
Let $u_{0} \in \hh^{-1} \cap L^{2}$ for $n=3$ and $u_{0}\in L^{2}$ for $n\ge 4$. Then, for any weak solution $u$ of (\ref{eq:3.17}), there exists $t_0>0$  such that
\[
\left\|\Lambda^{\zeta} u(t) \right\|_{L^{2}} \leq C_{\zeta} \left\|u(t_{0})\right\|_{\hh^{\beta_c}} (t-t_{0})^{-\frac{1}{2}(\zeta-\beta_{c})}, \quad \zeta>\beta_{c},
\]
where $C_\zeta \sim \zeta^\zeta C^\zeta$ and $\|u(t_{0})\|_{\hh^{\beta_c}}$ is sufficiently small to apply Theorem \ref{thm:2.8}.
\end{theorem}

\begin{remark}
In case $n=2$, the critical space is $\hh^{-\frac{1}{2}}$. If we show that 
\eqn \label{eq:3.20}
\liminf_{t \ra \infty}\|u\|_{\hh^{-\frac{1}{2}}}=0,
\een
decay of higher Sobolev norms hold. The higher order decay result for $n \ge 3$ appears to be new.
\end{remark}

\subsection{Cahn-Hilliard equation (special case)}
The Cahn-Hilliard equation  is given by
\eqn \label{eq:3.21}
u_t = - \Delta^2 u - \alpha \Delta u + \beta \Delta (u^3), \quad x \in \R^d, \ \beta >0, \ \alpha \ge 0.
\een
We here consider the special case when $\alpha =0$ and $\beta > 0$.  Since $n=3, \kappa=4, \alpha_{T_0}=2, \alpha_{T_i}=0, i=1,2,3$,  $\beta_{c}:=\frac{d}{2}-1$. We have the following decay result.

\begin{theorem} \label{thm:3.3}
Let $u$ be a (weak) solution of (\ref{eq:3.21}) with initial data $u_0 \in L^2$. Then, there exists $t_{0}>0$ such that
\[
\left\|\Lambda^{\zeta} u(t) \right\|_{L^{2}} \leq C_{\zeta} \left\|u(t_{0})\right\|_{\hh^{\beta_c}} (t-t_{0})^{-\frac{1}{2}(\zeta-\beta_{c})}, \quad \zeta>\beta_{c}
\]
where $C_\zeta \sim \zeta^\zeta C^\zeta$ when $p=2$ and $\|u(t_{0})\|_{\hh^{\beta_c}}$ is sufficiently small to apply Theorem \ref{thm:2.8}.
\end{theorem}

The generalized Cahn-Hilliard equation has been studied in \cite{temamdyn}. Local in time Gevrey regularity for for a particular case of this equation has been established in \cite{swanson11}. Due to Theorem \ref{thm:4.5}, we can in fact obtain a global Gevrey regularity result. Consequently,
 we can improve the previous decay result  to include the  generalized Cahn-Hilliard equation in the periodic case.

\begin{theorem} \label{thm:4.6}
Let $u:\mathbb{T}^{d} \ra \R, d=1,2,3$ be a solution of the general Cahn-Hilliard equation of the form
\[
u_t - \Delta^2 u = \Delta f(u), \quad f(u)=\sum_{j=1}^{2N-1}a_ju^j \quad \mbox{and} \quad a_{2N-1}=A >0.
\]
Assume further that
\eqn \label{eq:4.9}
\sum_{j=1}^{2N-2} j|a_j| < \delta,
\een
for $\delta $ suitably small and let $u_0 \in \hh^{\beta}$ with $ \frac{d}{2}< \beta < 2$. Then, any solution $u$ satisfies the decay estimate
\[
\left\|\Lambda^\zeta u(t)\right\|_{L^2} \le C_{\zeta} \|u(t_{0})\|_{\hh^{\beta}} (t-t_{0})^{-\frac{\zeta-\beta}{\kappa}}, \quad \zeta>\beta,
\]
where $\|u(t_{0})\|_{\hh^{\beta}}$ is sufficiently small to apply Theorem \ref{thm:4.5}.
\end{theorem}

\begin{remark}
The decay results presented in Theorems \ref{thm:3.2}, \ref{thm:3.3} and \ref{thm:4.6} are new to the best of our knowledge.
\end{remark}

\section{PROOFS OF MAIN RESULTS}

\subsection{Degree $n$ nonlinearity} 
We start with some preparatory results with the notation in Section 2.1. The following lemma is well-known, but we provide a proof for completeness and to give  precise estimates of the constants involved.

\begin{lemma}   \label{lem:2.1}
Let $\beta , t \ge 0 $ and $\kappa >0$. The kernels $k_1$ and $k_2$ corresponding to the symbols 
\[
m_1(\xi)=|\xi|^\beta e^{-t|\xi|_1^{\kappa}}, \quad m_2(\xi)=|\xi|^\beta e^{-t|\xi|^{\kappa}}
\]
are $L^1$ functions with
\[
\left\|k_i \right\|_{L^1} \le \frac{2C^{\beta/\kappa}\beta^{\beta/\kappa}}{t^{\beta/\kappa}}, \ \ i=1,2,
\]
where  $C$ is a constant independent of $\beta $ and $\kappa $.
\comments{is given in the proof below. In case $\kappa =2, p=2$, the operator norm of $k_i$ on $L^2$ is bounded by $\frac{C_{\beta}}{t^{\beta/\kappa}}$, where $C_{\beta} =\max\{1, \frac{\sqrt{n!}}{\sqrt{2}}\}$, with $n=\lceil x \rceil $ being the smallest integer
greater than or equal to $\beta$.}
\end{lemma}

\begin{proof}
We only need to obtain the desired estimate for $m_{1}$ and $m_{2}$ with $t=1$. Then, time dependent bounds can be obtain by the scaling: $\xi \mapsto t^{\frac{1}{\kappa}} \xi$. To estimate $m_{i}$ at $t=1$, we use the Littlewood-Paley decomposition (see Appendix for the definition of this decomposition). Since the estimation is the same for $m_{1}$ and $m_{2}$, we only estimate $m_{2}$. We apply $\Delta_{j}$ to $m_{2}$ and take the $L^{1}$ norm.
\[
\left\|\Delta_{j}m_{2} \right\|_{L^{1}} \le C 2^{j\beta} \left\|\Delta_{j} e^{-|\xi|^{\kappa}} \right\|_{L^{1}} \le C 2^{j\beta} e^{-2^{j\kappa}},
\]
where the first and second inequalities can be found in \cite{Hmidi}. Then, we have
\eqn  \label{eq:2.4}
\left\|m_{2} \right\|_{L^{1}} \le C \sum_{j\in Z} \left\|\Delta_{j}m_{2}\right\|_{L^{1}}\leq C \sum_{j\in Z} 2^{j\beta} e^{-2^{j\kappa}}\leq 2C^{\beta/\kappa}\beta^{\beta/\kappa}
\een
which completes the proof.
\end{proof}

\comments{
We set the right-hand side of (\ref{eq:2.4}) as $c_{\beta,\kappa}$. The second assertion can be computed in the Fourier space using Plancheral theorem, and the operator norm is given by
\[
\max_{|\xi| \ge 0}|\xi|^{\beta} e^{-|\xi|^{2}} \leq \max\{1, \max_{|\xi| \ge 1}|\xi|^{\beta} e^{-|\xi|^{2}} \}.
\]
Since $e^{|\xi|^{2}} \ge 1+ \frac{1}{n!} |\xi|^{2n}$,
\[
\max_{|\xi| \ge 1}|\xi|^{\beta} e^{-|\xi|^{2}} \le  \frac{\sqrt{n!}}{\sqrt{2}},
\]
where $n$ is the smallest integer such that $n \ge \beta$.}

From this lemma and the definitions of the operators $T_i$ and the Gevrey norms,  we immediately have the following estimates concerning the Gevrey norms.

\begin{lemma} \label{lem:2.2}
Denote $C_{\beta}=C^{\beta}\beta^{\beta}$. For any $s,t, \beta' \ge 0$, we have
\[
\left\|T_i v\right\|_{Gv(s,\beta,p)} \le C\|v\|_{Gv(s,\beta+\ati,p)}, \quad \left\|e^{-\frac{1}{2}t \Lambda^{k}}v \right\|_{Gv(s,\beta + \beta',p)} \le \frac{C_{\beta^{'}}}{t^{\beta' /\kappa}}\|v\|_{Gv(s,\beta,p)},
\]
where  $\ati$ is as defined in (\ref{eq:2.1}).
\end{lemma}

We also have the following result concerning the semigroup. To do so, we need the following elementary inequality:
\eqn \label{eq:2.5}
(x+y)^{\gamma} \le x^{\gamma} + y^{\gamma}, \quad x,y \ge 0, \quad \gamma \in [0,1].
\een
In the inequality above, by convention, we take $0^0 =1$ when $\gamma =0$ and   $\min(x,y)=0$.

\begin{lemma} \label{lem:2.3}
Let $0 \le s \le t < \infty$. Let
\[
E=e^{-c\left((t-s)^{1/\kappa}+s^{1/\kappa}- t^{1/\kappa}\right)\Lambda_1}.
\]
The operator $E$ is either the identity operator or is a Fourier multiplier with an $L^1$ kernel and its  $L^1$ norm is bounded independent of $s, t$.
\end{lemma}

\begin{proof}
Note that since $\kappa \ge 1$, by (\ref{eq:2.5}) we have $t^{1/\kappa} \le s^{1/\kappa} + (t-s)^{1/\kappa}$. Thus, $E = e^{-a\Lambda_1}$ where 
\[
a = c \left\{(t-s)^{1/\kappa}+s^{1/\kappa}- t^{1/\kappa} \right\} \ge 0.
\]
In case $a=0$, $E$ is the identity operator while if $a>0$, $E$ is a Fourier multiplier with symbol $m_E(\xi) = \prod_{i=1}^d e^{-a|\xi_i|}$. Thus, the kernel of $E$ is given by the product of one dimensional Poisson kernels $\prod_{i=1}^d\frac{a}{\pi(a^2+x_i^2)}$. The $L^1$ norm of this kernel is bounded by a constant independent of $a$.
\end{proof}

\begin{lemma}  \label{lem:2.4}
Let $\kappa \ge 1$. The operator $\tilde{E}=e^{-\frac{1}{2}a\Lambda^{\kappa}+a^{1/\kappa}c\Lambda_1}$ is a Fourier multiplier which acts as a bounded operator on all $L^p$ spaces ($1 < p < \infty$) and its operator norm is uniformly bounded with respect to $a \ge 0$.
\end{lemma}

\begin{proof}
When $a =0$, $\tilde{E}$ is the identity operator. On the other hand, if $a>0$, then $\tilde{E}$ is  Fourier multiplier with symbol $m_{\tilde{E}}(\xi)=e^{-\frac{1}{2}|a^{1/\kappa}\xi|^{\kappa} + c|a^{1/\kappa}\xi|_1}$. Since $m_{\tilde{E}}(\xi)$ is uniformly bounded for all $\xi$ and decays exponentially for $|\xi|>>1$ the claim follows from Hormander's multiplier theorem (see \cite{stein}).
\end{proof}

\begin{lemma}  \label{lem:2.5}
Let $u_0 \in \hh^{\beta_0}_p$ for some $\beta_0 \in \R$ and $0< T\le \infty, \ \beta >0$. Denote
\eqn \label{eq:2.6}
M= \sup_{0<t<T} t^{\beta/\kappa} \left\|e^{-t\Lambda^{\kappa}}u_0\right\|_{Gv(t,\beta_0+\beta,p)}.
\een
Then $M \le C_{\beta}\|u_0\|_{\hh^{\beta_0}_p}$ and $M \ra 0$ as $T \ra 0$.
\end{lemma}

\begin{proof}
Due to Lemmas \ref{lem:2.2}  and \ref{lem:2.4}, for any $0<t\le T$, we have
\begin{equation*}
 \begin{split}
 \left\|e^{-t\Lambda^{\kappa}}u_0 \right\|_{Gv(t,\beta_0+\beta,p)} &=  \left\|e^{ct^{1/\kappa}\Lambda_1-\frac{1}{2}t\Lambda^{\kappa}}\Lambda^{\beta}e^{-\frac{1}{2}t\Lambda^{\kappa}}\Lambda^{\beta_0}u_0 \right\|_{L^p} \le C \left\|\Lambda^{\beta}e^{-\frac{1}{2}t\Lambda^{\kappa}}\Lambda^{\beta_0}u_0 \right\|_{L^p} \\
 & \le C_{\beta}t^{-\beta/\kappa} \left\|\Lambda^{\beta_0}u_0\right\|_{L^p} =C_{\beta}t^{-\beta/\kappa}\|u_0\|_{\hh^{\beta_0}_p},
 \end{split}
\end{equation*}
which proves the first assertion. Concerning the second, note that there exists $u_0' \in \hh^{\beta+\beta_0}_p$ such that $\left\|u_0-u_0' \right\|_{\hh^{\beta_0}_p} < \epsilon $ for any $\epsilon >0$. Proceeding as above, and using the fact that $u_0' \in \hh^{\beta+\beta_0}_p$, we obtain
\begin{gather*}
 \begin{split}
 t^{\beta/\kappa} \left\|e^{-t\Lambda^{\kappa}}u_0 \right\|_{Gv(t,\beta_0+\beta,p)} & \le t^{\beta/\kappa} \left\|e^{-t\Lambda^{\kappa}}(u_0-u_0') \right\|_{Gv(t,\beta_0+\beta,p)}+ t^{\beta/\kappa} \left\|e^{-t\Lambda^{\kappa}}u_0' \right\|_{Gv(t,\beta_0+\beta,p)} \\
 & \le C\epsilon + Ct^{\beta/\kappa}\|u_0'\|_{\hh^{\beta+\beta_0}_p}.
 \end{split}
\end{gather*}
Letting $t \ra 0$ and noting that $\epsilon >0$ is arbitrary, the claim follows.
\end{proof}

We need the following versions of the Kato-Ponce inequality \cite{Kato2}. For $\Gamma = \Lambda $ or $\Gamma = (I+\Lambda)$, we have
\eqn \label{eq:2.7}
\left\|\Gamma^\beta(\phi \psi) \right\|_{L^p}  \le C\left[\left\|\Gamma^\beta \phi\right\|_{L^{p_1}} \left\|\psi\right\|_{L^{q_1}} + \left\|\phi \right\|_{L^{p_2}} \left\|\Gamma^\beta \psi \right\|_{L^{q_2}}\right],
\een
where $\beta \ge 0$ and $1 < p, p_{i}, q_{i} < \infty $ are such that $ \frac{1}{p}=\frac{1}{p_{1}}+\frac{1}{q_{1}}=\frac{1}{p_{2}}+\frac{1}{q_{2}}$. For completeness, a proof is provided in the Appendix. The following lemma, which shows that the Kato-Ponce inequality holds for Gevrey norms, is crucial to estimate the nonlinear term in (\ref{eq:1.1}).

\begin{lemma}  \label{lem:2.6}
Let $t, \beta \ge 0$ and $1 < p, p_{i}, q_{i} < \infty $ such that $ \frac{1}{p}=\frac{1}{p_{1}}+\frac{1}{q_{1}}=\frac{1}{p_{2}}+\frac{1}{q_{2}}$. Then, we have the following estimate:
\[
\left\|fg\right\|_{Gv(t,\beta,p)} \le C \left[ \left\|f \right\|_{Gv(t,\beta,p_1)} \left\|g \right\|_{Gv(t,0,q_1)} +\left\|f \right\|_{Gv(t,0,p_2)} \left\|g \right\|_{Gv(t,\beta,q_2)} \right].
\]
\end{lemma}

\begin{proof}
For notational convenience, denote
\begin{gather}  \label{phipsidef}
a = ct^{1/\kappa}, \quad \phi = e^{a\Lambda_1}f, \quad  \psi = e^{a\Lambda_1}g.
\end{gather}
By a density argument, it will be enough to prove the result for $\phi , \psi \in {\cal S}(\R^d)$. Note that, using the Fourier inversion formula,
we have
\begin{gather*}
\begin{split}
B_a(f,g) &:=e^{a\Lambda_1}(f \cdot g) =e^{a\Lambda_1} \left((e^{-a\Lambda_1}\phi)\cdot (e^{-a\Lambda_1}\psi)\right) \\
& = \frac{1}{(2\pi)^d}\int\!\int e^{i x \cdot(\xi+\eta)}e^{a \left(\|\xi + \eta\|_1-\|\xi\|_1-\|\eta\|_1\right)}\hat{f}(\xi)\hat{g}(\eta)\,
d\xi\, d\eta.
\end{split}
\end{gather*}
Recall that for a vector $\xi=(\xi_1,\cdots,\xi_d)$, we denoted $\|\xi\|_1=\sum_{i=1}^d |\xi_i|$. For $\xi=(\xi_1,\cdots,\xi_d), \eta = (\eta_1, \cdots , \eta_d)$, we now split the domain of integration of the above integral into sub-domains depending on the sign of $\xi_j, \eta_j$ and $\xi_j+\eta_j$. In order to do so, we introduce the operators acting on one variable  (see page 253 in \cite{LR}) by
\[
K_1f = \frac{1}{2\pi} \int_0^{\infty} e^{i x \xi}\hat{f}(\xi)\, d\xi, \quad K_{-1}f = \frac{1}{2\pi} \int_{-\infty}^{0} e^{i x \xi}\hat{f}(\xi)\, d\xi.
\]
Let the operators $L_{a,-1}$ and $L_{a,1}$ be defined by
\[
L_{a,1}f = f,\quad L_{a,-1}f = \frac{1}{2\pi}\int_{\R} e^{i x\xi}e^{-2a|\xi|}\hat{f}(\xi)\, d\xi.
\]
For $ \vec{\alpha}=(\alpha_1,\cdots , \alpha_d), \vec{\beta}  =(\beta_1,\cdots , \beta_d) \in \{-1,1\}^d$, denote
the operator
\[
Z_{a,\vec \alpha , \vec \beta} = K_{\beta_1}L_{t,\alpha_1 \beta_1} \otimes \cdots \otimes K_{\beta_d}L_{t,\alpha_d\beta_d} \quad  \mbox{and} \quad K_{\vec \alpha} = k_{\alpha_1} \otimes \cdots \otimes K_{\alpha_d}.
\]
The above tensor product  means that the $j-$th operator in the tensor product acts on the $j-$th variable of the function $f(x_1,\cdots,x_d)$. A tedious (but elementary) calculation now yields the following identity:
\[
B_a(f,g) = \sum_{\vec \alpha , \vec \beta , \vec \gamma \in \{-1,1\}^d}K_{\vec \alpha} \left(\left(Z_{a, \vec \alpha , \vec \beta}f \right) \cdot \left(Z_{a, \vec \alpha , \vec \gamma }g\right)\right).
\]
We now note that the operators $K_{\vec \alpha}, Z_{a, \vec \alpha , \vec \beta}$  are linear combinations of Fourier multipliers (including Hilbert transform)  and the identity operator, and commute with $\Lambda_1$ and $\Lambda$. Moreover, they are bounded linear operators on $L^p, 1 < p < \infty$ and the corresponding operator norm of $Z_{a,\vec \alpha, \vec \beta}$ is bounded independent of $a\ge 0$. Thus, we have
\begin{equation*}
 \begin{split}
 \left\|fg\right\|_{Gv(t,\beta,p)} &= \left\|\Lambda^\beta e^{a\Lambda_1}(fg)\right\|_{L^p}= \left\|\Lambda^\beta  e^{a\Lambda_1}\left(\left(e^{-a\Lambda_1}\phi \right) \left(e^{-a\Lambda_1}\psi \right)\right) \right\|_{L^p}\\
 & = \left\|\Lambda^\beta \sum_{\vec \alpha, \vec \beta, \vec \gamma \in \{-1,1\}^d}K_{\vec \alpha} \left( \left(Z_{a,\vec \alpha , \vec \beta}\phi\right)  \left(Z_{a,\vec \alpha , \vec \gamma}\psi \right) \right) \right\|_{L^p}\\
 & = \left\| \sum_{\vec \alpha, \vec \beta, \vec \gamma \in \{-1,1\}^d}K_{\vec \alpha}\Lambda^\beta \left (\left(Z_{a,\vec \alpha , \vec \beta}\phi \right)  \left(Z_{a,\vec \alpha , \vec \gamma}\psi \right) \right) \right\|_{L^p} \\
 & \le C \left\|\Lambda^\beta \left( \left(Z_{a,\vec \alpha , \vec \beta}\phi \right) \left(Z_{a,\vec \alpha , \vec \gamma}\psi \right) \right) \right\|_{L^p} \\
 &  \le C\left[ \left\|\Lambda^\beta Z_{a,\vec \alpha , \vec \beta}\phi \right\|_{L^{p_1}} \left\|Z_{a,\vec \alpha , \vec \gamma}\psi) \right\|_{L^{q_1}}
 + \left\|Z_{a,\vec \alpha , \vec \beta}\phi \right\|_{L^{p_2}} \left\|\Lambda^{\beta}Z_{a,\vec \alpha , \vec \gamma}\psi) \right\|_{L^{q_2}}\right] \\
 & \le C\left[ \left\| Z_{a,\vec \alpha , \vec \beta}\Lambda^\beta\phi \right\|_{L^{p_1}} \left\|Z_{a,\vec \alpha , \vec \gamma}\psi \right\|_{L^{q_1}}
 + \left\|Z_{a,\vec \alpha , \vec \beta}\phi \right\|_{L^{p_2}} \left\|Z_{a,\vec \alpha , \vec \gamma}\Lambda^\beta\psi \right\|_{L^{q_2}}\right] \\
 & \le  C \left[ \left\|f \right\|_{Gv(t,\beta,p_1)} \left\|g \right\|_{Gv(t,0,q_1)} +\left\|f \right\|_{Gv(t,0,p_2)} \left\|g \right\|_{Gv(t,\beta,q_2)} \right],
 \end{split}
\end{equation*}
where to obtain the above relations, we used the commutativity of $\Lambda^\beta $ with the operators
$K_{\vec \alpha}$ and $Z_{a,\vec \alpha, \vec \beta}$, the $L^p$-boundedness of these operators with uniformly bounded operator norms  with respect to $a\ge 0$ and (\ref{eq:2.7}).
\end{proof}

\begin{lemma} \label{lem:2.7}
Let $s \ge 0,\, 1 < p < \infty$ and $0< \alpha, \ \beta < \frac{d}{p}$ and  $\alpha + \beta > \frac{d}{p}$. We have
\eqn \label{eq:2.8}
\left\|fg\right\|_{Gv(s,\gamma,p)} \le C\left\|f\right\|_{Gv(s,\alpha,p)} \left\|g\right\|_{Gv(s,\beta,p)}, \quad \gamma = \alpha + \beta - \frac{d}{p}.
\een
\end{lemma}

\begin{proof}
Applying Lemma \ref{lem:2.6}, we have
\eqn  \label{eq:2.9}
\left\|fg\right\|_{Gv(s,\gamma,p)} \le C \left( \left\|f\right\|_{Gv(s,\gamma,r_1)} \left\|g\right\|_{Gv(s,0,r_2)} + \left\|f\right\|_{Gv(s,0,r_3)} \left\|g\right\|_{Gv(s,\gamma,r_4)}\right),
\een
where $\frac{1}{r_1}+\frac{1}{r_2}=\frac{1}{r_3}+\frac{1}{r_4} = \frac{1}{p}$. We now need the Sobolev inequalities (\cite{LR})
\eqn \label{eq:2.10}
\|f\|_{L^q} \le C \left\|\Lambda^{\delta}f\right\|_{L^p}, \quad 0 \le \delta < \frac{d}{p}, \quad \delta = \frac{d}{p} - \frac{d}{q}.
\een
By taking $r_i, i=1,\cdots,4$ in (\ref{eq:2.9}) with
\[
\frac{1}{r_1}= \frac{1}{p}- \frac{\alpha-\gamma}{d}, \quad \frac{1}{r_2}=\frac{1}{p}-\frac{\beta}{d}, \quad \frac{1}{r_3} = \frac{1}{p} - \frac{\alpha}{d}, \quad \frac{1}{r_4}= \frac{1}{p} - \frac{\beta-\gamma}{d},
\]
and applying the Sobolev inequalities given above, we obtain (\ref{eq:2.8}).
\end{proof}

An iterative application of this lemma yields the following:
\begin{equation} \label{eq:2.11}
\begin{split}
& \left\|\prod_{i=1}^n f_i \right\|_{Gv(s,\gamma,p)} \le  C\prod_{i=1}^n \left\|f_i\right\|_{Gv(s,\alpha_i,p)}, \\
& \gamma = \sum_{i=1}^n\alpha_i - \frac{(n-1)d}{p} > 0, \ \  \max_{1\leq i \leq n}\{\alpha_i\} < \frac{d}{p}.
\end{split}
\end{equation}

\subsubsection*{\bf Proof of Theorem \ref{thm:2.8}.}
Let  $\gamma = \beta_0 + \beta$ for adequate $\beta >0$ to be specified later and define the Banach space
\[
E = \left\{ u\in C([0,T); \hh^{\beta_0}_p):   \|u\|_E := \max\{\|u\|_{E_1}, \|u\|_{E_2} \} < \infty  \right\},
\]
where
\[
\|u\|_{E_1}:=\sup_{s \in (0,T)}\|u(s)\|_{Gv(s,\beta_0,p)}, \quad \|u\|_{E_2}:=s^{\beta/\kappa}\|u(s)\|_{Gv(s,\gamma,p)}.
\]
Additionally,
\[
{\cal E} = \left\{u \in E: \left\|u-e^{-t\Lambda^{\kappa}}u_{0} \right\|_E \le M \right\},
\]
where $M$ is as in (\ref{eq:2.6}). It is easy to see that ${\cal E}$ is a compete metric space; it is a closed and bounded subset of the Banach space E. For $S$  as  in (\ref{eq:1.2}),  note that 
\eqn \label{eq:2.15}
(Su)(t) = e^{-t\Lambda^\kappa}u_0+(Bu)(t),  \quad (Bu)(t) = \int_0^t e^{-(t-s)\Lambda^\kappa}G(u(s))\, ds.
\een
We need to prove an inequality of the form
\eqn   \label{eq:2.16}
\left\|(Bu)\right\|_{E} \le CT^{\mu} \|u\|_{E_{2}}^n
\een
for $\mu \ge 0$ satisfying $\mu =0$ if and only if $\beta_0 = \beta_c$. Note that this implies
\[
\left\|Su-e^{-t\Lambda^\kappa}u_0\right\|_{E} \le  CT^{\mu}M^n, \quad \left\|Su-Sv \right\|_{E}\le nCT^{\mu}M^{n-1}\|u-v\|_E.
\]
The first inequality  immediately  follows from (\ref{eq:2.16}) while the second follows by noting
\eqn  \label{eq:2.17}
G(u)-G(v)= \sum_{i =1}^nT_0F \left(T_1u, \cdots, T_{i-1}u, T_i(u-v),T_{i+1}v, \cdots,T_n(v) \right).
\een
If $\beta_0 > \beta_c$, then $\mu >0$ and  $S$ is contractive in ${\cal E}$ provided $T$ is  sufficiently small. On the other hand, in case $\beta_0=\beta_c$ and $\mu=0$, in view of Lemma \ref{lem:2.5}, we  can either choose  $T$ small in case the initial data is arbitrary or the initial data sufficiently small if $T=\infty$,  to reach the same conclusion. We now proceed to estimate $\|Bu\|_{E_2}$, the estimate for  $\|Bu\|_{E_1}$ being similar. Note that
\begin{equation} \label{eq:2.18}
 \begin{split}
 & \left\|e^{-(t-s)\Lambda^\kappa}G(u(s)) \right\|_{Gv(t,\gamma,p)} = \left\|e^{ct^{1/\kappa} \Lambda_1}\Lambda^{\gamma}e^{-(t-s)\Lambda^\kappa}G(u(s)) \right\|_{L^p} \\
 & \le C \left\|e^{c(t-s)^{1/\kappa}\Lambda_1}e^{cs^{1/\kappa}\Lambda_1}e^{-(t-s)\Lambda^\kappa}\Lambda^{\gamma}G(u(s)) \right\|_{L^p}  \\
 & = C \left\|e^{c(t-s)^{1/\kappa}\Lambda_1-\frac{1}{2}(t-s)\Lambda^\kappa}e^{cs^{1/\kappa}\Lambda_1}\Lambda^{\gamma} e^{-\frac{1}{2}(t-s)\Lambda^\kappa}G(u(s)) \right\|_{L^p} \\
 & \le C \left\|\Lambda^{\gamma}e^{cs^{1/\kappa}\Lambda_1} e^{-\frac{1}{2}(t-s)\Lambda^\kappa}G(u(s)) \right\|_{L^p}.
 \end{split}
\end{equation}
By (\ref{eq:2.11}) with $\alpha_i = \gamma - \ati$ and Lemma \ref{lem:2.2} with $\beta' = \sum_{i=0}^n \ati + (n-1) \frac{d}{p}-(n-1)\gamma$, 
\begin{equation} \label{eq:2.19}
 \begin{split}
 & t^{\beta/\kappa} \left\|(Bu)(t) \right\|_{Gv(t,\gamma,p)} \\
 & \le C \|u\|_{E_2}^nt^{\beta/\kappa} \int_0^t\frac{1}{s^{(n\beta)/\kappa}}\frac{1}{(t-s)^{\frac{1}{\kappa}\left[\sum_{i=0}^n \ati + (n-1)  \frac{d}{p}-(n-1)\gamma\right]}} ds \\
 & \le C\|u\|_{E_2}^nt^{\mu}, \quad \mu = \frac{(n-1)(\beta_0 - \beta_c)}{\kappa}.
 \end{split}
\end{equation}
In order to apply (\ref{eq:2.11}) and Lemma \ref{lem:2.2}, and to ensure the finiteness of the integral in (\ref{eq:2.19}), we need
\[
\beta' = \sum_{i=0}^n \ati + (n-1)\left(\frac{d}{p} -\gamma \right) \ge 0, \ \  \gamma - \min_{1\le i \le n}\ati < \frac{d}{p}, \ \  n\gamma - \sum_{i=1}^n\ati - (n-1)\frac{d}{p} >0.
\]
For the convergence of the integral in (\ref{eq:2.19}), we also need $n\beta < \kappa$ and $\beta' < \kappa$. A choice of $\beta>0$ satisfying all these conditions can be made provided the conditions on the parameters stated in the theorem hold.

\subsection{Analytic Nonlinearity}
Here we follow the notation in subsection 2.2. We once again start with some auxiliary results.

\begin{lemma}   \label{lem:4.1}
Let $\beta > \frac{d}{p}$. For any $t \ge 0$, we have
\[
\|f g\|_{Gv(t,\beta,p)} \le C\|f\|_{Gv(t,\beta,p)} \|g\|_{Gv(t,\beta,p)}.
\]
\end{lemma}

\begin{proof}
For notational convenience, denote 
\[
a = ct^{1/\kappa}, \quad \phi = e^{a\Lambda_1}f, \quad \psi = e^{a\Lambda_1}g.
\]
Proceeding as in the proof of Lemma \ref{lem:2.6}  (with the notation there in) and using (\ref{eq:4.4}), we obtain that 
\begin{equation*}
 \begin{split}
 & \|f g\|_{Gv(t,\beta,p)} \\
 &\le C\left[ \left\| Z_{a,\vec \alpha , \vec \beta}(I+\Lambda)^\beta\phi \right\|_{L^{p}} \left\|Z_{a,\vec \alpha , \vec \gamma}\psi \right\|_{L^{\infty}} + \left\|Z_{a,\vec \alpha , \vec \beta}\phi \right\|_{L^{\infty}} \left\|Z_{a,\vec \alpha , \vec \gamma}(I+\Lambda)^\beta\psi \right\|_{L^{p}}\right] \\
 & \le C \left[ \left\|(I+\Lambda)^\beta\phi \right\|_{L^{p}} \left\|(I+\Lambda)^{\beta}Z_{a,\vec \alpha , \vec \gamma}\psi \right\|_{L^p} + \left\|(I+\Lambda)^\beta Z_{a,\vec \alpha , \vec \beta}\phi \right\|_{L^p} \left\|(I+\Lambda)^\beta\psi \right\|_{L^{p}}\right] \\
 & \le C\|\phi\|_{\h^{\beta}_p}\|\psi\|_{\h^{\beta}_p} \le C\|f \|_{Gv(t,\beta,p)}\|g\|_{Gv(t,\beta,p)},
 \end{split}
\end{equation*}
where we have used the fact that $(I+\Lambda)$ commutes with $Z_{a,\vec \alpha , \vec \beta}$.
\end{proof}

\begin{lemma}   \label{lem:4.2}
Let $\beta > \frac{d}{p}, s \ge 0$ and  $u, v$ be such that
\[
\max_{1\leq i \leq n} \left\{ \left\|T_i u\right\|_{Gv(s, \beta , p)}, \left\|T_i v\right\|_{Gv(s,\beta,p)} \right\} \le R.
\]
For notational simplicity, denote $F(T_1u, \cdots,T_n u)=F(u)$.  Then, there exists a  constant $C$ independent of $s, u, v$ such that
\begin{equation} \label{eq:4.6}
 \begin{split}
 & \left\|F(u) \right\|_{Gv(s, \beta,p)} \le F_M(RC),  \\
 & \left\|F(u)-F(v)\right\|_{Gv(s, \beta,p)}\le CF_M'(RC) \max_{1\leq i \leq n} \left\|T_i(u-v) \right\|_{Gv(s, \beta,p)}.
 \end{split}
\end{equation}
\end{lemma}

\begin{proof}
The first inequality in (\ref{eq:4.6}) is an immediate consequence of Lemma \ref{lem:4.1}. Concerning the second,  proceeding as in (\ref{eq:2.17}), we have
\begin{equation*}
 \begin{split}
 \left\|F(u) - F(v)\right\|_{Gv(s, \beta,p)} & \le \sum_{j \in \Z_+^d}|a_j||j|C^{|j|}R^{|j|-1} \max_{1\le i \le n} \left\|T_i(u-v) \right\|_{Gv(s,\beta,p)}  \\
 & \le CF_M'(RC)\max_{1\le i \le n} \left\|T_i(u-v) \right\|_{Gv(s,\beta,p)}.
 \end{split}
\end{equation*}
This completes the proof.
\end{proof}

\subsubsection*{\bf Proof of Theorem \ref{thm:4.3}.}
We recall $\alpha= \displaystyle \max_{1 \leq i \leq n}\{\ati\} $, and set
\[
E = \Big\{u \in C\big((0,T); \h^{\beta_0+\alpha} \big) : \|u\|_E = \sup_{s \in (0,T)}\|u(s)\|_{Gv(s, \alpha, p)}  \le  2R \Big\}.
\]
Proceeding as in the proof of Theorem \ref{thm:2.8} and using Lemma \ref{lem:4.2}, we obtain
\eqn \label{eq:4.7}
\left\|e^{-(t-s)\Lambda^\kappa }G(u(s)) \right\|_{Gv(t, \alpha, p)} \le \frac{C}{(t-s)^{(\alpha_{T_0}+ \alpha)/\kappa}} F_M(2RC).
\een
Thus, we obtain that 
\[
\left\|\int_0^{t} e^{-(t-s)\Lambda^\kappa }G(u(s))\, ds \right\|_E \le C F_M(2RC) T^{1 - \frac{\alpha_{T_0}+\alpha}{\kappa}}.
\]
Similarly, using the second inequality in (\ref{eq:4.6}), we can also obtain
\[
\left\|Su - Sv\right\|_E \le C F_M'(2RC)T^{1 - \frac{\alpha_{T_0}+ \alpha }{\kappa}}  \|u-v\|_E.
\]
The proof is now completed along the lines of the proof of Theorem \ref{thm:2.8}.

\comments{\begin{remark}
Our result is better than \cite{Ferrari} in the sense that the nonlinearity $G$ in (\ref{eq:4.3}) is more general and
we cover all $p$, not simply $p=2$.
\end{remark}
}

\subsection{The periodic case}
 In this case, $\Lambda $ has a minimum eigenvalue, denoted by $\lambda_0 >0$, and the Fourier spectrum of all periodic functions with space average zero is contained in the complement of a ball with radius $\lambda_0$. Thus, following \cite{Danchin, Hmidi}, we can show that
\[
\left\|e^{-a\Lambda^\kappa }u_0 \right\|_{L^p} \le Ce^{-a\lambda_0^\kappa}\|u_0\|_{L^{p}}, \quad a>0.
\]
This fact can be easily proven for $p=2$ using Plancheral theorem. Thus, instead of (\ref{eq:4.7}), we can obtain
\begin{equation} \label{eq:4.8}
 \begin{split}
 & \left\|e^{-(t-s)\Lambda^\kappa }G(u(s)) \right\|_{Gv(t, \beta+ \alpha, p)} =  \left\|e^{-\frac{1}{4}(t-s)\Lambda^\kappa }e^{-\frac{3}{4}(t-s)\Lambda^\kappa }G(u(s)) \right\|_{Gv(t, \beta+\alpha, p)}\\
 & \le e^{-\frac12\lambda_0^\kappa (t-s)} \left\|e^{-\frac{3}{4}(t-s)\Lambda^\kappa }G(u(s))\right\|_{Gv(t, \beta+\alpha, p)}  \le \frac{Ce^{-\frac12\lambda_0^\kappa (t-s)}}{(t-s)^{(\alpha_{T_0}+ \alpha )/\kappa}} F_M(2RC).
 \end{split}
\end{equation}
Using now the elementary fact (see for instance Proposition 7.5 in \cite{biswas2007gevrey})
\[
\sup_{t>0} \int_0^t e^{-b(t-s)}s^{-a}ds  \le C < \infty, \quad b>0, \quad 0<a<1,
\]
it follows that
\[
\left\|Su\right\|_{E} \le CF_M(2RC) \quad  \mbox{and} \quad  \left\|Su-Sv\right\|_E \le CF_M'(2RC) \|u - v\|_E.
\]
The proof of Theorem \ref{thm:4.5} follows  from the above discussion and by noting that
\[
\lim_{R \ra 0}F_M(2RC)=0, \quad \limsup_{R\ra 0} F_M'(2RC) \le \delta.
\]
We can thus ensure that $S$ is contractive  for $T=\infty$ provided $\|\bu_0\|_{\hh^{\beta +\alpha}_p}$ and $\delta $ are small.

\subsection{Equation in Fourier space.}
As before, we start with some estimates on Gevrey norms.

\begin{lemma}   \label{lem:4.7}
For any $s \ge 0$ and $\beta =0$, we have
\[
\|u \ast v\|_{Gv(s)} \le C\|u\|_{Gv(s)}\|v\|_{Gv(s)}.
\]
\end{lemma}

\begin{proof}
Let $\gamma = \frac{1}{2}s^{1/\kappa}$. By triangle inequality $|\xi| \le |\xi - \eta|+|\eta|$,
\begin{equation*}
 \begin{split}
 \|u \ast v\|_{Gv(s)} & \le C \int_{\R^{d_1}} \int_{\R^{d_1}} e^{\gamma |\xi|}|u|(\xi-\eta)|v|(\eta)d\xi d\eta\\
 & \le C  \int_{\R^{d_1}} \int_{\R^{d_1}} e^{\gamma |\xi - \eta|}|u|(\xi-\eta)e^{\gamma |\eta|}|v|(\eta)d\xi d\eta \le C\|u\|_{Gv(s)}\|v\|_{Gv(s)},
 \end{split}
\end{equation*}
which completes the proof.
\end{proof}

Lemma \ref{lem:4.7} tells that for any $s\ge0$, $Gv(s)$ is a Banach algebra. Therefore, this space can be used to estimate analytic functions $G(u)$ as follows.

\begin{lemma}   \label{lem:4.8}
Let $G$ be as in (\ref{eq:4.3}) and  $u, v$ be such that
\[
\max_{1\leq i \leq n} \left\{\left\|T_iu\right\|_{Gv(s, \beta , p)}, \left\|T_iv\right\|_{Gv(s,\beta,p)} \right\} \le R.
\]
For notational simplicity, denote $F(T_1u, \cdots,T_nu)=F(u)$.  Then, there exists a  constant $C$ independent of $s, u, v$ such that
\begin{equation} \label{eq:4.13}
 \begin{split}
 & \left\|F(u)\right\|_{Gv(s)} \le C_1F_M(RC),\\
 & \left\|F(u)-F(v)\right\|_{Gv(s)}\le C F^{'}_{M}(RC) \max_{1\leq i \leq n} \left\|T_i(u-v)\right\|_{Gv(s)}.
 \end{split}
\end{equation}
\end{lemma}

\begin{proof}
The first inequality in (\ref{eq:4.13}) is an immediate consequence of Lemma \ref{lem:4.7}. Concerning the second, we have
\begin{equation*}
 \begin{split}
 F(u) - F(v) & = \sum_{j\in \Z_+^d}a_j \left(u ^{*j}-v^{*j}\right)  \\
 &= \sum_{j \in \Z_+^d} a_j \sum_{k=0}^{|j|-1} \left[ u^{*(|j|-|k|)} \ast v^{*|k|} - u^{*(|j|-|k|-1)}\ast v^{*(|k|+1)} \right].
 \end{split}
\end{equation*}
Applying Lemma \ref{lem:4.7}, we readily obtain
\begin{equation*}
 \begin{split}
 \left\|F(u) - F(v)\right\|_{Gv(s)} & \le \sum_{j \ge 1}|a_j|jC^jR^{j-1} \max_{1\leq i\leq n} \left\|T_{i}(u -v)\right\|_{Gv(s)} \\
 &\le C F^{'}_{M}(RC) \max_{1\leq i\leq n} \left\|T_{i}(u -v)\right\|_{Gv(s)}.
 \end{split}
\end{equation*}
This completes the proof.
\end{proof}

\begin{lemma} \label{lem:4.9}
From the assumptions on $T_i$, for $ s, \ t, \ \beta \ge 0$,
\eqn \label{eq:4.14}
\left\|T_iv \right\|_{Gv(s)} \le \|v\|_{Gv(s, \ati)}, \quad \left\|T_ie^{-\frac{1}{2}t\fd}v \right\|_{Gv(s)} \le \frac{C_{\ati}}{t^{\ati /\kappa}}\|v\|_{Gv(s)}.
\een
\end{lemma}

\begin{proof}
By definition of the Gevery norm, the first term in (\ref{eq:4.14}) can be obtained by
\[
\left\|T_iv\right\|_{Gv(s)} \le \int_{\R^{d}} |\xi|^{\ati} e^{\frac{1}{2}s^{\frac{1}{\kappa}}|\xi|} |v(\xi)|d\xi =\|\bv\|_{Gv(s, \ati)}.
\]
Since for all $\alpha \ge 0$
\[
|\xi|^{\alpha} e^{-t|\xi|^{\kappa}} \leq C_{\alpha} t^{-\frac{\alpha}{\kappa}}, \quad \text{ where C is independent of} \quad \xi, t, \alpha,
\]
we can prove the second term in (\ref{eq:4.15}) by
\begin{equation*}
 \begin{split}
 \left\|T_ie^{-\frac{1}{2}t\fd}v\right\|_{Gv(s)} &\le \int_{\R^{d}} |\xi|^{\ati} e^{-\frac{1}{2}|\xi|^{\kappa}t} e^{\frac{1}{2}s^{\frac{1}{\kappa}}|\xi|} |v(\xi)|d\xi \\
 & \leq \frac{C_{\ati}}{t^{\ati /\kappa}} \int_{\R^{d}} e^{\frac{1}{2}s^{\frac{1}{\kappa}}|\xi|} |v(\xi)|d\xi =\frac{C}{t^{\ati /\kappa}} \|v\|_{Gv(s)}.
 \end{split}
\end{equation*}
This completes the proof.
\end{proof}

\subsubsection*{\bf Proof of Theorem \ref{thm:4.10}}
We set
\[
E = \Big\{u\in C\big((0,T); \bbv \big):\|u\|_E = \sup_{s \in (0,T)}\max  \left\{\|u(s)\|_{Gv(s)}, \ \|u(s)\|_{Gv(s, \alpha)} \right\} \le  R \Big\}.
\]
For $u \in E$, we define
\eqn  \label{eq:4.15}
(Su)(t) = e^{-t\fd}u_0 + \int_0^{t} e^{-(t-s)\fd}G(u(s)) ds.
\een
Note that since $\kappa \ge 1$, we have $t^{1/\kappa} \le s^{1/\kappa} + (t-s)^{1/\kappa}$. Consequently,
\begin{equation} \label{eq:4.16}
 \begin{split}
 & \|e^{-(t-s)\fd}G(u(s))\|_{Gv(t)}  \le \int_{\R^{d}} e^{\frac{1}{2}t^{1/\kappa}|\xi|}e^{-(t-s)|\xi|^{\kappa}} \left|(T_{0}F(u(s)))(\xi) \right| d\xi  \\
 & \le C \int_{\R^{d}} e^{\frac{1}{2}(t-s)^{1/\kappa}|\xi|}e^{-\frac{1}{2}(t-s)|\xi|^{\kappa}}e^{\frac{1}{2}s^{1/\kappa}}e^{-\frac{(t-s)}{2}|\xi|^{\kappa}}  \left|(T_{0}F (u(s)))(\xi) \right| d\xi  \\
 & \le C \int_{\R^{d}} e^{\frac{1}{2}s^{1/\kappa}|\xi|}e^{-\frac{(t-s)}{2}|\xi|^{\kappa}} \left|(T_{0}F(u(s)))(\xi) \right| d\xi,
 \end{split}
\end{equation}
where we use the fact that
\eqn  \label{eq:4.17}
e^{\frac{1}{2}\tau^{1/\kappa}|\xi|}e^{-\frac{\tau}{2}|\xi|^{\kappa}} =
e^{\frac{1}{2}|\tau^{1/\kappa}\xi|}e^{-\frac{1}{2}|\tau^{1/\kappa}\xi|^{\kappa}} \le C
\een
with the constant $C$ independent of $\tau, \xi$. Therefore,
\begin{equation} \label{eq:4.18}
 \begin{split}
 \left\|e^{-(t-s)\fd}G(u(s))\right\|_{Gv(t)} & \leq \frac{C_{\alpha_{T_{0}}}}{(t-s)^{\alpha_{T_0} /\kappa}} \left\|F(u)\right\|_{Gv(t)} \le \frac{C_{\alpha_{T_{0}}}}{(t-s)^{\alpha_{T_0}/\kappa}} C_1F_M(RC),
 \end{split}
\end{equation}
where we use (\ref{eq:4.14}) to obtain the first  inequality, and we use the first inequality in (\ref{eq:4.13}) to obtain the second inequality in (\ref{eq:4.18}). Similarly,  one can obtain also the estimate
\eqn \label{eq:4.19}
\left\|e^{-(t-s)\fd}G(u(s)) \right\|_{Gv(t, \alpha)} \le \frac{C_{\alpha_{T_0}+ \alpha}}{(t-s)^{(\alpha_{T_0}+ \alpha)/\kappa}} C_1F_M(RC).
\een
Thus, we obtain that 
\begin{equation} \label{eq:4.20}
 \begin{split}
  \left\|\int_0^{t} e^{-(t-s)\fd}G(u(s))\, ds \right\|_E  \le C_{\alpha_{T_0}+ \alpha} F_M(RC)\max \left\{ T^{1 - \frac{\alpha_{T_0}+ \alpha}{\kappa}}, T^{1 - \frac{\alpha_{T_0}}{\kappa}} \right\}.
 \end{split}
\end{equation}
Proceeding along similar lines but using the second inequality in (\ref{eq:4.13}) in the last step, we can also obtain
\eqn \label{eq:4.21}
\left\|Su_1 - Su_2 \right\|_E \le F^{'}_{M}(RC) \max \left\{ T^{1 - \frac{\alpha_{T_0}+ \alpha}{\kappa}}, T^{1 - \frac{\alpha_{T_0} }{\kappa}} \right\} \|u_1-u_2\|_E.
\een
Concerning the linear term, using (\ref{eq:4.17}), it is easy to see that
\eqn \label{eq:4.22}
\max \left\{ \|e^{-t\fd}u_0\|_{Gv(t)}, \|e^{-t\fd}u_0\|_{Gv(t, \alpha)} \right\} \le \tilde{C}\max \left\{ \|u_0\|, \|u_0\|_{\alpha} \right\} \le \frac{R}{2}.
\een
From (\ref{eq:4.20}), (\ref{eq:4.21}) and (\ref{eq:4.22}), it follows that $S$ is contractive in $E$ provided $T$ is suitably small. Thus we can find a fixed point. The adjustment in the above argument necessary to obtain the global result in  the periodic setting is similar to Subsection 4.3 above.

\section{PROOF OF APPLICATIONS}

We now provide the proof of the decay results stated in Section 3. 

\subsection{Proof of Theorem \ref{thm:3.1}.}
It is enough to prove (\ref{eq:3.1}). For $u_{0} \in L^{2}$, we have the following energy estimate
\[
\|u(t)\|_{L^2}^2 + \int_0^t\|u(s)\|_{\hh^1}^2\, ds \le \|u_0\|_{L^2}^2.
\]
This implies that
\eqn \label{eq:3.4}
\sup_{t >0}\|u(t)\|_{L^2}^2 \le \|u_0\|_{L^2}^2, \quad \liminf_{t \ra \infty}\|u(t)\|_{\hh^1} =0.
\een
In order to obtain the second relation in (\ref{eq:3.4}), for $\epsilon >0$  arbitrary, choose $t$ large so that $\frac{1}{t}\|u_0\|_{L^2}^2 < \epsilon /4$. We note that the energy inequality yields 
\[
\frac{1}{t}\int_0^t\|u(s)\|_{\hh^1}^2\, ds \le \frac{1}{t}\|u_0\|_{L^2}^2.
\]
This immediately implies that there exists $t_0 \in (0,t)$ such that $\|u(t_0)\|_{\hh^1}^2 < \epsilon $. Recall now the interpolation inequality
\[
\|u\|_{\hh^{\beta}} \le \|u\|_{\hh^{\beta_1}}^{\theta}\|u\|_{\hh^{\beta_2}}^{1-\theta}, \quad \beta = \theta\beta_1 +(1-\theta)\beta_2, \quad \theta \in (0,1), \quad \beta_i \in \R,  \ i=1,2.
\]
Due the uniform bound on $\|u(t)\|_{L^2}^2$, it also follows that $\displaystyle\liminf_{t \ra \infty}\|u(t)\|_{\hh^{\beta}} =0$ for $0 < \beta \le 1$. When $p=2$, we have $\beta_c= \frac{1}{2}$ and consequently (\ref{eq:3.1}) holds, which implies (\ref{eq:3.3}). For $p >2$, (\ref{eq:3.3}) follows from a direct application of the Sobolev inequalities (\ref{eq:2.10}). We now focus on the case $1<p < 2$. To this end, we use the  the vorticity $\omega = \nabla \times u$. Note first that
\[
\liminf_{t\ra \infty}\|\omega\|_{L^2}^2=\liminf_{t\ra \infty}\|\nabla u\|_{L^2}^2  = \liminf_{t\ra \infty}\|u\|_{\hh^1}^2 =0.
\]
From the vorticity equation, 
\[
\omega_{t} +u\cdot \nabla \omega -\Delta \omega =\omega \nabla u,
\]
we have  the uniform $L^1$ bound (see \cite{cf})
\eqn  \label{eq:3.5}
 \|\omega(t)\|_{L^1} \le C(\|u_0\|_{L^2}^2 + \|\omega_0\|_{L^1}).
\een
The uniform $L^1$ bound and the interpolation inequality
\[
 \|\omega\|_{L^q} \le \|\omega\|_{L^1}^\theta\|\omega\|_{L^2}^{(1-\theta)},  \quad 1/q=\theta + (1-\theta)/2
\]
implies
\eqn  \label{eq:3.6}
 \liminf_{t\ra \infty} \|\omega\|_{L^q} =0, \quad 1<q\le 2.
\een
Recall that $\nabla u = \nabla (\nabla \times (\Delta)^{-1}\omega)$. The operator $\nabla (\nabla \times (\Delta)^{-1})$
is a Fourier multiplier of homogeneous degree zero, and consequently, by  the Calderon-Zygmund theorem, we have
\eqn  \label{eq:3.7}
\|\omega\|_{\hh^\zeta_q}= \left\|\Lambda^\zeta \omega \right\|_{L^q} \sim  \left\|\Lambda^\zeta\nabla u \right\|_{L^q} \sim \|u\|_{\hh^{1+\zeta}_q}, \quad 1 < q < \infty , \quad \zeta \ge 0.
\een
Consequently,
\eqn \label{eq:3.8}
 \liminf_{t\ra \infty}\|u\|_{\hh^{1}_q}=0, \quad 1<q<2.
\een
For $p=3/2, \ \beta_c =1$ and (\ref{eq:3.1}) follows directly from (\ref{eq:3.8}).  For $3/2 < p <2$ and $\beta_c$ as above, applying (\ref{eq:2.10}), we  have 
\[
\|u\|_{\hh^{\beta_c}_p} \le \|u\|_{\hh^1_{3/2}}
\]
and once again (\ref{eq:3.1}) follows from (\ref{eq:3.8}). We now consider the case $1<p < 3/2$. From (\ref{eq:3.7}) and (\ref{eq:3.3}) with $p=2$ which we already established, it follows that
\eqn  \label{eq:3.9}
\lim_{t\ra \infty}\|\omega\|_{\hh^{\zeta}}\le \lim_{t\ra \infty}\|u\|_{\hh^{1+\zeta}}=0, \quad \zeta \ge 0.
\een
Recall now the generalization of the Gagliardo-Nirenberg inequalities for the fractional homogeneous  Sobolev spaces (see \cite{meyer2006oscillating}), namely,
\eqn \label{eq:3.10}
\|\omega\|_{\hh^{\alpha}_p}\le C \|\omega\|_{\hh^{\alpha_1}_{q_1}}^{\theta}\|\omega\|_{L^{q_2}}^{1-\theta},\quad \theta \in (0,1), \ \alpha = \theta \alpha_1, \quad \frac{1}{p}   = \frac{\theta}{q_1}+ \frac{1-\theta}{q_2}.
\een
Combining (\ref{eq:3.10}) with (\ref{eq:3.6}) and (\ref{eq:3.9}) we have
\[
\liminf_{t \ra \infty} \|\omega\|_{\hh^{\beta}_p} =0, \quad \beta >0, \quad 1<p<2.
\]
Consequently, by (\ref{eq:3.7}), we also have
\[
\liminf_{t \ra \infty} \|u\|_{\hh^{1+\beta}_p} =0, \quad 1<p<3/2, \quad \beta >0.
\]
Noting that for $1<p<3/2$, we have $\beta_c >1$. By taking $\beta$ such that $1+\beta=\beta_{c}$, (\ref{eq:3.1}) follows.

\subsection{Proof of Theorem \ref{thm:qgdecay}.}
For notational simplicity, we will write $\|\cdot \|_{Gv(t,0,p_0)} = \|\cdot \|_{Gv(t)}$. Note that in the notation of Theorem \ref{thm:2.8}, we have $T_0 = \nabla$,  \ $T_1 = {\cal R}$, and $ T_2 =I$. It is known that solution to (\ref{qg}) satisfies (see \cite{cordoba2004})
\[
\lim_{t\ra \infty} \|\eta\|_{L^p} =0, \quad 1 \le p \le \infty .
\]
We apply the local existence part of Theorem \ref{thm:2.8} with 
\[
d=2, \ n=2, \ \alpha_{T_0}=1, \ \alpha_{T_1}=\alpha_{T_2}=0, \ p_0 = \frac{2}{\kappa -1}, \ \beta_c =0. 
\]
Let $t_1$ be such that $\|\eta(t_1)\|_{L^{p_0}} < \epsilon $ where $\epsilon $ is as in Theorem \ref{thm:2.8}.
Applying the global existence part of this theorem, we have
\begin{gather*}
\sup_{t > t_1} \left\|\eta(t)\right\|_{Gv(t-t_1)} < \infty.
\end{gather*}
From (\ref{gevdecay}), we obtain
\[
\left\|\eta(t) \right\|_{\hh^{\alpha}_{p_0}} \le \frac{C^{\alpha} \alpha^{\alpha}}{(t-t_1)^{\frac{\alpha}{\kappa}}}
\le \frac{C_1^{\alpha} \alpha^{\alpha}}{t^{\frac{\alpha}{\kappa}}}\ \mbox{for all} \ t\ge t_1+1, \alpha >0.
\]
Also from the local existence part of the theorem, there exists $t_2>0, \beta >0$ such that
\begin{gather}  \label{applylocal}
\max \left\{ \|\eta\|_{Gv(t)}, t^{\beta/\kappa}\|\eta\|_{\hh^{\beta}_{p_0}} \right\} \le 2\|\eta_0\|_{L^{p_0}} \ \mbox{for all} \ 0<t \le t_2.
\end{gather}
Thus (\ref{qgdecay}) holds for all $t \in (0,t_2] \cup [t_1+1 , \infty)$. To complete the proof, we will need to show (\ref{qgdecay}) for $t \in [t_2,t_1+1]$. In fact, since $t$ lies in the compact interval $[t_2, t_1+1]$, it will be enough to show an estimate of the form
\begin{gather}  \label{decaymodified}
\left\|\eta(t) \right\|_{\hh^{\alpha}_{p_0}} \le C^{\alpha} \alpha^{\alpha}
\end{gather}
for a constant $C$ that does not depend on $t$ or $\alpha $. Note that due to (\ref{applylocal}), we have $\|\eta(t_2/2)\|_{\hh^{\beta}_{p_0}} < \infty $. Due to the global well-posedness for the sub-critical quasi-geostrophic equations, we have (see \cite{wu2001, Carrillo})
\begin{gather*}
M:= \sup_{t_2 \le t \le t_1+1}\|\eta(t)\|_{\hh^{\beta}_{p_0}} < \infty .
\end{gather*}
Thus, applying the non-critical case of Theorem \ref{thm:2.8}, there exists a time $0< t_3 < t_2/2$ depending on $M$, such that
\[
\left\|\eta(t) \right\|_{Gv(t_3)} \le 2\|\eta(t-t_3)\|_{\hh^{\beta}_{p_0}} \le 2M \ \mbox{for all}\ t_2 \le t \le t_1+1.
\]
Once again due to (\ref{gevdecay}), this yields (\ref{decaymodified}).

\subsection{Proof of Theorem \ref{thm:nonlinheat}} 
As in the previous cases, multiplying by $u$, integrating by parts and applying the Sobolev inequality, we have
we obtain
\begin{gather*}
\dt \|u(t)\|_{L^2}^2 + \left\|\Lambda^{\kappa/2}u \right\|_{L^2}^2 \le C|\alpha| \left\|\Lambda^{\kappa/2}u \right\|_{L^2}^2.
\end{gather*}
In order to apply the Sobolev inequality, we implicitly use condition (i) or (ii). Since $|\alpha| < \epsilon $, this immediately implies
\begin{gather}  \label{nonlinheatineq}
\|u(t)\|_{L^2}^2 + \int_0^T \left\|\Lambda^{\kappa/2}u(s) \right\|_{L^2}^2\, ds\le \|u_0\|_{L^2}^2.
\end{gather}
Using (\ref{nonlinheatineq}) for the Galerkin approximation, one can establish existence of a global weak solution of (\ref{nonlinheat}) satisfying (\ref{nonlinheatineq}). From (\ref{nonlinheatineq}), it follows as before that 
\[
\liminf_{t \ra \infty} \|\Lambda^{\kappa/2}u\|_{L^2} =0.
\]
 Again from (i) or (ii), it follows that $\beta_c \le \kappa/2$ and the Poincar\'{e} inequality implies
that 
\[
\liminf_{t \ra \infty} \|\Lambda^{\beta_c}u\|_{L^2} =0.
\]
The rest of the proof follows as in the previous cases. Concerning the whole space, the proof is the same when we proceed using (\ref{eq:2.10}).

\subsection{Proof of Theorem \ref{thm:3.2}.}
As noted previously, it suffice to prove (\ref{eq:3.1}). We begin with the the $L^{2}$ energy estimate. We multiply (\ref{eq:3.17}) by $u$ and integrate over $\R$. Using integration by parts, we get
\[
\|u(t)\|_{L^2}^2 + \int_0^t\|u(s)\|_{\hh^1}^2\, ds \le \|u_0\|_{L^2}^2.
\]
As in the previous subsection, this implies 
\[
\liminf_{t \ra \infty}\|u(t)\|_{\hh^1} =0.
\]
Consequently, due to the uniform bound on $\|u\|_{L^2}^2$, it also follows that 
\[
\liminf_{t \ra \infty}\|u(t)\|_{\hh^{\beta}} =0
\]
for all $0 < \beta \le 1$. If $n\ge 4$, $\beta_{c}\in (0,1)$ and so the decay now follows. For $n=3$, the critical space is $L^{2}$. Therefore, we need to show 
\[
\liminf_{t\ra \infty}\|u(t)\|_{L^2}=0.
\]
To do this, it will be enough to obtain a time independent bound for $\|u(t)\|_{\hh^{-1}}$. To this end, we multiply (\ref{eq:3.17}) by $\Lambda^{-2}u$ and integrate by parts over $\R$ to get
\begin{equation} \label{eq:3.18}
 \begin{split}
 & \frac{d}{dt} \left\|\Lambda^{-1}u(t) \right\|^{2}_{L^{2}} +\|u(t)\|^{2}_{L^{2}} =\int_{\R} \Lambda^{-2}u(t) \partial_{x}({u(t)}^{3})dx\\
 & \leq \left\|\Lambda^{-1} u(t)\right\|_{L^{2}} \|u(t)\|_{L^{2}} \|u(t)\|^{2}_{L^{\infty}}  \leq C \left[ \left\| \Lambda^{-1}u(t) \right\|^{2}_{L^{2}} \|u(t)\|^{4}_{L^{\infty}}\right] +\frac{1}{2} \|u(t)\|^{2}_{L^{2}},
 \end{split}
\end{equation}
where we have used the fact that the Hilbert transform $\partial_x \Lambda^{-1}$ is a bounded operator on  $L^2$. In one dimension, we have the inequality \cite{Friedman}:
\eqn \label{eq:3.19}
\|u(t)\|_{L^{\infty}} \leq C \|u(t)\|^{\frac{1}{2}}_{L^{2}} \left\|\nabla u(t)\right\|^{\frac{1}{2}}_{L^{2}}.
\een
Using (\ref{eq:3.18}), (\ref{eq:3.19}) and Gronwall's inequality, we obtain
\begin{equation*}
 \begin{split}
 & \left\|\Lambda^{-1} u(t) \right\|^{2}_{L^{2}} \\
 &\leq C \left\|\Lambda^{-1}u_{0}\right\|^{2}_{L^{2}} \exp \int^{t}_{0} \|u(s)\|^{4}_{L^{\infty}} ds  \leq C \left\|\Lambda^{-1} u_{0}\right\|^{2}_{L^{2}} \exp \int^{t}_{0} \|u(s)\|^{2}_{L^{2}} \|\nabla u(s)\|^{2}_{L^{2}} ds\\
 & \leq C \left\| \Lambda^{-1} u_{0} \right\|^{2}_{L^{2}} \exp \left[ \|u_{0}\|^{2}_{L^{2}} \int^{t}_{0}\|\nabla u(s)\|^{2}_{L^{2}} ds\right] \leq C \left\|\Lambda^{-1}u_{0}\right\|^{2}_{L^{2}} \exp \left[ \|u_{0}\|^{4}_{L^{2}}\right].
 \end{split}
\end{equation*}
This finishes the proof.

\subsection{Proof of Theorem \ref{thm:3.3}.}
We only need to show (\ref{eq:3.1}). Multiplying (\ref{eq:3.21}) by $u$ and integrating by parts, we arrive at
\[
\frac{1}{2}\dt \|u\|_{L^2}^2 + \left\|\Delta u\right\|_{L^2}^2 \le - 3\beta \int_{R^{d}} |u|^2 \left|\nabla u\right|^2dx.
\]
Noting $\beta>0$, as before, this immediately yields
\[
\|u(t)\|_{L^2}^2 + \int_0^t \left\|\Delta u(s)\right\|_{L^2}^2 ds \le \|u_0\|_{L^2}^2.
\]
This yields an time independent uniform bound for $\|u(t)\|_{L^2}$ and also that 
\[
\liminf_{t\ra \infty}\|\Delta u\|_{L^2}=0.
\]
Interpolation immediately yields (\ref{eq:3.1}) in case $d \ge 3$. For cases $d=1,2$, multiplying (\ref{eq:3.21}) by $\Lambda^{-2}u$ (recall $\Lambda^2=-\Delta$) and integrating by parts, we arrive at
\[
\frac{1}{2}\dt \left\|\Lambda^{-1}u\right\|_{L^2}^2 + \left\|\Lambda u\right\|_{L^2}^2 \le - 3\beta \int_{R^{d}} |u|^4 dx .
\]
This yields
\[
\|u\|_{\hh^{-1}} \le \|u_0\|_{\hh^{-1}}, \quad \liminf_{t\ra \infty} \|u\|_{\hh^1}=0.
\]
For $d=1$ and $d=2$, we have $\beta_c= -1/2$ and $\beta_c=0$ respectively, and thus we have (\ref{eq:3.1}).

\subsection{Proof of Theorem \ref{thm:4.6}.}
Multiplying the equation by $u$ and integrating by parts, we have the energy inequality
\begin{equation*}
 \begin{split}
 &\frac12 \dt \|u\|_{L^2}^2 + \left\|\Delta u\right\|_{L^2}^2 \\
 & \le \sum_{j=1}^{2N-2} j|a_j|\int_{\mathbb{R}^{d}} |\nabla u|^2 |u|^{j-1}dx - (2N-1)a_{2N-1}\int_{\mathbb{R}^{d}} |\nabla u|^2 u^{2(N-1)}dx.
 \end{split}
\end{equation*}
Let
\[
I_1 = \sum_{j=1}^{2N-2} j|a_j| \int_{u:|u| \le 1} \left|\nabla u\right|^2 |u|^{j-1}dx, \quad I_2 = \sum_{j=1}^{2N-2} j|a_j|\int_{u:|u| > 1} |\nabla u|^2 |u|^{j-1}dx.
\]
By applying (\ref{eq:4.9}) and Poincare inequality, we get
\[
I_1 \le \delta \int_{\mathbb{R}^{d}} \left|\nabla u\right|^2dx \le \delta \lambda_0 \|u\|_{L^2}^2, \quad I_2 \le \delta \int_{\mathbb{R}^{d}} |\nabla u|^2 u^{2(N-1)}dx.
\]
Provided $\delta $ is sufficiently small, from the energy inequality, we get
\[
\frac12 \dt \|u\|_{L^2}^2 + \gamma \left\|\Delta u\right\|_{L^2}^2 \le  0,
\]
where $\gamma >0$ is an adequate constant. As before, from here we obtain
\[
\|u\|_{L^2}^2 \le \|u_0\|_{L^2}^2, \quad \limsup_{t \ra \infty} \left\|\Delta u\right\|_{L^2} =0.
\]
By interpolation, this implies that 
\[
\liminf_{t \ra \infty}\|u\|_{\hh^{\beta}} =0
\]
for all $\beta <2$. By Poincare inequality, this in turn implies 
\[
\liminf_{t \ra \infty}\|u\|_{L^2} =0
\]
as well. Consequently, 
\[
\liminf_{t \ra \infty}\|u\|_{\h^{\beta}} =0
\]
for all $\beta < 2$. This finishes the proof.

\section{APPENDIX} \label{sec:5}

We now present the Littlewood-Paley theory and its application to the Kato-Ponce inequality. For more details of the Littlewood-Paley theory, see \cite{chemin1998perfect, Danchin}.

\subsection{Littlewood-Paley theory}
We take a couple of smooth functions $(\chi, \phi)$ supported on $\{\xi ; |\xi|\leq 1\}$ with values in $[0,1]$ such that for all $\xi \in \mathbb{R}^{d}$,
\[
\chi(\xi)+ \displaystyle\sum^{\infty}_{j=0}\psi(2^{-j}\xi)=1,
\]
where $\psi(\xi)=\phi\left(\xi/2\right)-\phi(\xi)$. We denote $\psi\left(2^{-j}\xi\right)$ by $\psi_{j}(\xi)$. The homogeneous dyadic blocks and lower frequency cut-off functions are defined by
\[
\Delta_{j}u=2^{jd} \int_{\mathbb{R}^{d}} h\left(2^{j}y\right)u(x-y)dy, \quad S_{j}u=2^{jd} \int_{\mathbb{R}^{d}} \tilde{h}\left(2^{j}y\right)u(x-y)dy,
\]
where $h=\cf^{-1}\psi$ and $\tilde{h}=\cf^{-1}\chi$. We note that 
\[
u=\sum_{j\in \mathbb{Z}} \Delta_{j}u \quad \text{in} \quad {\cs}^{'}_{h},
\]
where ${\cs}^{'}_{h}$ is the space of tempered distributions $u$ such that $ \displaystyle \lim_{j\rightarrow -\infty}S_{j}u=0$ in ${\cs}^{'}$. This is called the Littlewood-Paley decomposition. This decomposition allows us to characterize a large range of functions spaces in a unified way. In this section, we only consider the homogeneous Triebel-Lizorkin spaces (\cite{Frazier}):
\[
\|f\|_{\dot{F}^{\alpha}_{p,q}}=\left\| \Big(\sum_{j\in\mathbb{Z}} 2^{2j\alpha}|\Delta_{j}f|^{2} \Big)^{\frac{1}{2}}\right\|_{L^{p}}.
\]
In particular, for $q=2$,
\eqn  \label{eq:5.1}
\|f\|_{\hh^{\alpha}_p} \simeq \left\| \Big(\sum_{j\in\mathbb{Z}}2^{2j\alpha}|\Delta_{j}f|^{2}\Big)^{\frac{1}{2}} \right\|_{L^{p}}.
\een

The concept of paraproduct is to deal with the interaction of two functions in terms of low or high frequency parts. For $u$, $v$ two tempered distributions,
\begin{equation*}
 \begin{split}
 & uv=T_{u}v+T_{v}u+R(u,v), \quad \text{where}\\
 & T_{u}v= \sum_{i\leq j-2}\Delta_{i}u\Delta_{j}v= \sum_{j\in\mathbb{Z}}S_{j-1}u\Delta_{j}v, \quad S_{j}u= \sum_{l\leq j-1}\Delta_{l}u,  R(u,v)= \sum_{|j-j^{'}|\leq 1}\Delta_{j}u \Delta_{j^{'}}v.
 \end{split}
\end{equation*}
Then, up to finitely many terms,
\[
\Delta_{j}\left(T_{u}v\right)=S_{j-1}u\Delta_{j}v, \quad \Delta_{j}\left(R(u,v)\right)=\sum_{k\ge j-2}\Delta_{k}u\Delta_{k}v.
\]

\subsection{Proof of the Kato-Ponce inequality \cite{Kato2}} 

We only prove the homogeneous part: for $p$, $p_{i}$, and $q_{i}$ such that $
\frac{1}{p}=\frac{1}{p_{1}}+\frac{1}{q_{1}}=\frac{1}{p_{2}}+\frac{1}{q_{2}}$, $1\leq p<\infty$, $p_{i}, q_{i} \ne 1$, we have the following estimation:
\eqn \label{eq:5.2}
\left\|\Lambda^{s}(fg)\right\|_{L^{p}} \le C \left[ \left\|\Lambda^{s}f\right\|_{L^{p_{1}}} \|g\|_{L^{q_{1}}} +\|f\|_{L^{p_{2}}} \left\|\Lambda^{s}g\right\|_{L^{q_{2}}} \right].
\een

\begin{proof}
We decompose $fg$ as follows:
\begin{equation} \label{eq:5.3}
 \begin{split}
 fg &=\sum_{k} \Big(\sum_{j \leq k-2}\Delta_{j}f \Big) \Delta_{k}g + \sum_{k} \Big(\sum_{j \leq k-2}\Delta_{j}g \Big)\Delta_{k}f  +\sum_{|j-k|\leq 1} \Delta_{j}f \Delta_{k}g\\
 &=\sum_{k}S_{k-1}f \Delta_{k}g +\sum_{k} \Delta_{k} S_{k-1}g \Delta_{k} f +\sum_{|j-k|\leq 1} \Delta_{j}f \Delta_{k}g\\
 &=(a)+(b)+(c).
 \end{split}
\end{equation}
We apply $\Lambda^{s}$ to (\ref{eq:5.3}) and estimate three terms separately. We begin with (a).
\begin{equation} \label{eq:5.4}
 \begin{split}
 \left\|\Lambda^{s}(a) \right\|_{L^{p}} &\le C \left\| \Big[ \sum_{k} \Big| \Lambda^{s}\Big( \sum_{l} S_{l-1}f \Delta_{l}g\Big) \Big|^{2}\Big]^{\frac{1}{2}} \right\|_{L^{p}} \\
 &\le C \left\| \Big[ \sum_{k} \big| 2^{ks} S_{k-1}f \Delta_{k}g \big|^{2}\Big]^{\frac{1}{2}}\right\|_{L^{p}}.
 \end{split}
\end{equation}
Since for any $k$,
\[
\left|S_{k-1}f(x)\right| \leq C\cm f(x)=\sup_{r>0} \frac{1}{r^{d}}\int_{B(x,r)}|f(y)|dy,
\]
where $\cm$ is the Hardy-Littlewood maximal operator, we can estimate the right-hand side of (\ref{eq:5.4}) by
\begin{equation} \label{eq:5.5}
 \begin{split}
 \left\|\Lambda^{s}(a) \right\|_{L^{p}} & \le C \left\| \cm f \left[ \sum_{k} \big|2^{ks}\Delta_{k}g \big|^{2}\right]^{\frac{1}{2}}\right\|_{L^{p}} \le C \left\|\cm f\right\|_{L^{p_{1}}} \left\| \left[ \sum_{k} \big|2^{ks} \Delta_{k} g \big|^{2} \right]^{\frac{1}{2}} \right\|_{L^{q_{1}}} \\
 & \le C \|f\|_{L^{p_{1}}} \left\|\Lambda^{s}g\right\|_{L^{q_{1}}},
 \end{split}
\end{equation}
where we use the fact that $\cm$ maps $L^{p}$ to $L^{p}$ for all $p>1$. By using the same method,
\eqn \label{eq:5.6}
\left\|\Lambda^{s}(b)\right\|_{L^{p}} \le C \left\|\Lambda^{s}f\right\|_{L^{p_{2}}} \|g\|_{L^{q_{2}}}.
\een
We finally estimate $\Lambda^{s}(c)$.
\begin{equation} \label{eq:5.7}
 \begin{split}
 \left\|\Lambda^{s}(c)\right\|_{L^{p}} &\le C \left\| \Big[ \sum_{k} \Big| \Lambda^{s} \Big( \sum_{|l-l^{'}| \leq 1} \Delta_{l} f \Delta_{l^{'}}g\Big)\Big|^{2}\Big]^{\frac{1}{2}}\right\|_{L^{p}} \\
 & \le C \left\| \Big[ \sum_{k} \Big| 2^{ks} \sum_{l \ge k-2} \Delta_{l} f  \Delta_{l}g \Big|^{2}\Big]^{\frac{1}{2}}\right\|_{L^{p}}\\
 & = C \left\| \left[ \left\{ 2^{ks} \Big( \sum_{l \ge k-2} \Big( 2^{-l s} \Big)^{2} \Big)^{\frac{1}{2}} \times \Big( \sum_{l} \Big( \Delta_{k} (\Delta_{l}f 2^{ls} \Delta_{l}g) \Big)^{2} \Big)^{\frac{1}{2}} \right\}^{2} \right]^{\frac{1}{2}} \right\|_{L^{p}} \\
 & \le C \left\| \Big[ \sum_{k} \sum_{l} \Big( \Delta_{k} \big( \Delta_{l}f 2^{ls} \Delta_{l}g \big) \Big)^{2}\Big]^{\frac{1}{2}} \right\|_{L^{p}}.
 \end{split}
\end{equation}
We now need to use the extension of the Littlewood-Paley operator \cite{LR}: if $\mathbb{L}: L^{p} \rightarrow L^{p}l^{2}$ is a Littlewood-Paley operator, then $\cl: L^{p}l^{2} \rightarrow L^{p}l^{2}l^{2}$ is the extension of $\mathbb{L}$ such that
\[
\left\| \cl\right\|_{L^{p}l^{2} \rightarrow L^{p}l^{2}l^{2}} \le C \left\| \mathbb{L}\right\|_{L^{p} \rightarrow L^{p}l{2}}.
\]
Using this relation, we can replace the last term in (\ref{eq:5.7}) by
\begin{equation} \label{eq:5.8}
 \begin{split}
 \left\|\Lambda^{s}(c)\right\|_{L^{p}} & \le C \left\| \Big[  \sum_{l} \Big( \Delta_{l}f 2^{ls} \Delta_{l}g \Big)^{2}\Big]^{\frac{1}{2}} \right\|_{L^{p}} \le C \left\| \Big[  \cm f \sum_{l} \Big(2^{ls} \Delta_{l}g \Big)^{2}\Big]^{\frac{1}{2}} \right\|_{L^{p}}\\
 & \le C \left\|\cm f \right\|_{L^{p_{2}}} \left\| \Big[ \sum_{l}\Big( 2^{ls}\Delta_{l}g\Big)^{2}\Big]^{\frac{1}{2}}\right\|_{L^{q}_{2}} \le C \|f\|_{L^{p_{2}}} \left\|\Lambda^{s}g\right\|_{L^{q_{2}}}.
 \end{split}
\end{equation}
By (\ref{eq:5.5}), (\ref{eq:5.6}), and (\ref{eq:5.8}), we obtain (\ref{eq:5.2}).
\end{proof}

\section*{Acknowledgments}
H.B. was partially supported by NSF grants DMS10-08397 and FRG07-57227 while A.B. was supported by the NSF grant DMS-1109532. H. B. also gratefully acknowledges the support by the Center for Scientific Computation and Mathematical Modeling (CSCAMM)  at University of Maryland  and the Department of Mathematics at the University of California Davis where this research was performed.

\end{document}